\documentclass[a4paper,12pt]{amsart}

\usepackage{amsmath}
\usepackage{amssymb}
\usepackage{mathrsfs}
\usepackage{ifthen}
\usepackage{graphicx}
\usepackage{mathtools}
\usepackage{enumitem}
\usepackage{tikz-cd}
\usepackage[T1]{fontenc} %skandit
\allowdisplaybreaks

\def\switchlinenumbers{\@ifstar
    {\let\makeLineNumberOdd\makeLineNumberRight
     \let\makeLineNumberEven\makeLineNumberLeft}%
    {\let\makeLineNumberOdd\makeLineNumberLeft
     \let\makeLineNumberEven\makeLineNumberRight}%
    }

\def\setmakelinenumbers#1{\@ifstar
  {\let\makeLineNumberRunning#1%
   \let\makeLineNumberOdd#1%
   \let\makeLineNumberEven#1}%
  {\ifx\c@linenumber\c@runninglinenumber
      \let\makeLineNumberRunning#1%
   \else
      \let\makeLineNumberOdd#1%
      \let\makeLineNumberEven#1%
   \fi}%
  }

\nonstopmode \numberwithin{equation}{section}
\newtheorem*{theorem*}{Theorem}

\newtheorem{thm}{Theorem}[section]
\newtheorem{cor}[equation]{Corollary}
\newtheorem{lem}{Lemma}[section]
\newtheorem{prop}[equation]{Proposition}

\theoremstyle{definition}
\newtheorem{defn}{Definition}[section]
\newtheorem{example}{Example}[section]

\newtheorem{prob}[equation]{Problem}
\newtheorem{rem}{Remark}[section]

\newcounter{minutes}
\divide\time by 60
\newcounter{hours}
\multiply\time by 60
\addtocounter{minutes}{-\time}
\newcounter {own}
\def\theown {\thesection       .\arabic{own}}

\newenvironment{pf}[1][]{%
 \vskip 3mm
 \noindent
 \ifthenelse{\equal{#1}{}}%
  {{\slshape Proof. }}%
  {{\slshape #1.} }%
 }%
{\qed\bigskip}

\newcounter{alphabet}

\def\be{\begin{equation}}
\def\ee{\end{equation}}

\newcommand{\bee}{\begin{enumerate}}
\newcommand{\eee}{\end{enumerate}}

\newcommand{\blem}{\begin{lem}}
\newcommand{\elem}{\end{lem}}
\newcommand{\bthm}{\begin{thm}}
\newcommand{\ethm}{\end{thm}}
\newcommand{\bcor}{\begin{cor}}
\newcommand{\ecor}{\end{cor}}
\newcommand{\beg}{\begin{examp}}
\newcommand{\eeg}{\end{examp}}
\newcommand{\begs}{\begin{examples}}
\newcommand{\eegs}{\end{examples}}
\newcommand{\bdefe}{\begin{defin}}
\newcommand{\edefe}{\end{defin}}
\newcommand{\bprob}{\begin{prob}}
\newcommand{\eprob}{\end{prob}}
\newcommand{\bei}{\begin{itemize}}
\newcommand{\eei}{\end{itemize}}
\newcommand{\real}{{\operatorname{Re}\,}}

\newcommand{\norm}[1]{\left\lVert#1\right\rVert}

\newcommand{\innpdct}[1]{\left\langle#1\right\rangle}

\makeatletter
\renewcommand\p@enumi{\thetheorem\,}
\makeatother

\begin{document}

\title{Composition-differentiation operators on Hardy-Hilbert space of Dirichlet series }

\author{Vasudevarao Allu}
\address{Vasudevarao Allu,
Department  of Mathematics, School of Basic Sciences,
Indian Institute of Technology Bhubaneswar,
Bhubaneswar-752051, Odisha, India.}
\email{avrao@iitbbs.ac.in}

\author{Dipon Kumar Mondal}
\address{Dipon Kumar Mondal,
Department  of Mathematics, School of Basic Sciences,
Indian Institute of Technology Bhubaneswar,
Bhubaneswar-752051, Odisha, India.}
\email{diponkumarmondal@gmail.com}

\subjclass[{AMS} Subject Classification:]{Primary 47B33, 47B38, 30B50; Secondary , 30H10}

\def\thefootnote{}
\footnotetext{ {\tiny File:~\jobname.tex,
printed: \number\year-\number\month-\number\day,
          \thehours.\ifnum\theminutes<10{0}\fi\theminutes }
} \makeatletter\def\thefootnote{\@arabic\c@footnote}\makeatother

\begin{abstract}
	In this paper, we establish a compactness criterion for the composition-differentiation operator $D_\Phi$ on the Hardy space $\mathcal{H}^2$ of Dirichlet series in terms of a boundary decay condition for its mean counting function. Via a comparison-type principle, we show that this boundary decay is equivalent to a corresponding decay condition for the Green's function, which we further analyze through harmonic measure. We provide explicit mapping properties of the symbol $\Phi$ that generate a bounded composition-differentiation operator $D_\Phi$ and obtain precise norm estimates for $D_\Phi$ when $\Phi$ is an affine symbol with a single-prime in the class $\mathcal{G}_0$. Furthermore, we establish explicit upper and lower bounds for the approximation numbers of $D_\Phi$ on $\mathcal{H}^2$ motivated by the work of Queffélec and Seip [J. Funct. Anal., 2015]. Finally, we investigate spectral and operator-theoretic properties of $D_\Phi$ for symbols in $\mathcal{G}_0$.
\end{abstract}

\allowdisplaybreaks
\maketitle
\pagestyle{myheadings}
\markboth{Vasudevarao Allu, Dipon Kumar Mondal}{Composition-differentiation operators on $\mathcal{H}^2$-space of Dirichlet series}

\section{Introduction}
For a real number $\theta$, we denote the open right half-plane by $\mathbb{C}_\theta = \{s \in \mathbb{C} \colon \operatorname{Re}(s) > \theta\}$, where $\operatorname{Re}(s)$ denotes the real part of $s \in \mathbb{C}$. Let $\mathbb{D}$ refer to the open unit disk in the complex plane $\mathbb{C}$, and let $\operatorname{Hol}(\Omega)$ denote the algebra of holomorphic functions defined on a domain $\Omega \subseteq \mathbb{C}$. A Dirichlet series is a series of function of the form $f(s) = \sum_{n=1}^\infty a_n n^{-s}$, where $s = \sigma + it$ is a complex variable with $\sigma = \operatorname{Re}(s)$. The Hardy space of Dirichlet series $\mathcal{H}^2$ was introduced and thoroughly investigated by Gordon and Hedenmalm \cite{Gordon-Michigan-1999} in their foundational paper. This space serves as the natural Dirichlet-series analogue of the classical Hardy space on the unit disk $\mathbb{H}^2(\mathbb{D})$. The space $\mathcal{H}^2$ contains the $\ell^2$-inner product of the coefficients, where $\ell^2$ space represents the Hilbert space of square summable complex sequences. For holomorphic maps $f(s)= \sum_{n=1}^{\infty}a_n n^{-s}$ and $g(s)= \sum_{n=1}^{\infty}b_n n^{-s}$ in $\mathcal{H}^2$, their inner product is defined as $\innpdct{f, g}= \sum_{n=1}^{\infty}a_n\overline{b_n}$. The Hardy-Hilbert space of Dirichlet series is defined as
  \begin{equation}
  	\mathcal{H}^2 := \left\{f(s)= \sum_{n=1}^{\infty}a_n n^{-s} :\norm{f}_{\mathcal{H}^2}^2 = \sum_{n=1}^{\infty}|a_n|^2<\infty\right\}.
  \end{equation}
  
By the Cauchy-Schwarz inequality, any Dirichlet series $\sum_{n=1}^\infty a_n n^{-s}$ with $(a_n)_{n=1}^\infty \in \ell^2$ converges absolutely in the half-plane $\mathbb{C}_{1/2}$. Consequently, the Hardy space $\mathcal{H}^2$ consists of all Dirichlet series $f(s) = \sum_{n=1}^\infty a_n n^{-s}$ on $\mathbb{C}_{1/2}$ endowed with the Hilbert space norm $\|f\|_{\mathcal{H}^2} = \left(\sum_{n=1}^\infty |a_n|^2\right)^{1/2}$. As shown by Gordon and Hedenmalm \cite{Gordon-Michigan-1999}, functions in $\mathcal{H}^2$ generally do not admit holomorphic extensions to any domain strictly containing $\mathbb{C}_{1/2}$. For comprehensive background on Dirichlet series, we refer to \cite[Chapter~11]{Apostol-Book-1976}. The classical Hardy-Hilbert space on the unit disk $\mathbb{D}$ is defined as
 \begin{equation*}
 	\mathbb{H}^2(\mathbb{D}) := \left\{f\in\mbox{Hol}(\mathbb{D}) :\norm{f}_{\mathbb{H}^2(\mathbb{D})}^2 =\sup_{0\le r<1}\int_{\mathbb{T}}|f(r\xi)|^2 dm(\xi) = \sum_{n=1}^{\infty}|a_n|^2<\infty\right\},
 \end{equation*}
 where the inner product between two functions $f$ and $g$ in $\mathbb{H}^2(\mathbb{D})$ is defined as
 \begin{equation*}
 	\innpdct{f,g} := \lim_{r\to 1^-}\frac{1}{2\pi}\int_{0}^{2\pi} f(re^{it})\overline{g(re^{it})} dt
 \end{equation*} 
 and $m$ is the normalized Lebesgue measure on $\mathbb{T}$. By Fatou's radial limit theorem \cite{Shapiro-Book-1992}, the classical Hardy space $\mathbb{H}^2(\mathbb{D})$ is naturally identified with a closed subspace of $L^2(\mathbb{T})$. In the setting of Dirichlet series, Bohr's correspondence \cite[Section~2]{Hedenmalm-Duke-1997} establishes that $\mathcal{H}^2$ is isometric to the Hardy space on the infinite polycircle, $\mathbb{H}^2(\mathbb{T}^\infty)$, where 
 \begin{equation*}
 	\mathbb{T}^\infty = \{z = (z_1, z_2, \dots) \colon |z_j| = 1 \text{ for all } j \in \mathbb{N}\}
 \end{equation*}
 is the infinite Cartesian product of unit circles $\mathbb{T}$ equipped with normalized Haar measure $m_\infty$. The polycircle $\mathbb{T}^\infty$ is a compact abelian group whose dual group is $\mathbb{Z}^{(\infty)}$, consisting of all eventually zero sequences $\alpha = (\alpha_j)_{j \ge 1}$ of integers. Let $\mathbb{N}_0^{(\infty)}$ denote the subset of $\mathbb{Z}^{(\infty)}$ comprising sequences with $\alpha_j \ge 0$ for all $j \ge 1$, where $\mathbb{N}_0 :=\mathbb{N}\cup\{0\}$. A function $f \in L^2(\mathbb{T}^\infty)$ belongs to $\mathbb{H}^2(\mathbb{T}^\infty)$ if, and only if, its Fourier coefficients $\widehat{f}(\alpha)$ vanish for all non-positive frequencies, that is,
 \begin{equation*}
 	\widehat{f}(\alpha) = \int_{\mathbb{T}^\infty} f(z) \overline{z}^\alpha \, dm_\infty(z) = 0 \quad \text{for all } \alpha \in \mathbb{Z}^{(\infty)} \setminus \mathbb{N}_0^{(\infty)},
 \end{equation*}
 where $z^\alpha = \prod_{j \ge 1} z_j^{\alpha_j}$ for $z \in \mathbb{T}^\infty$. Any function $f \in L^2(\mathbb{T}^\infty)$ admits the Fourier expansion
 \begin{equation*}
 	f = \sum_{\alpha \in \mathbb{Z}^{(\infty)}} \widehat{f}(\alpha) e_\alpha \quad \text{with} \quad \sum_{\alpha \in \mathbb{Z}^{(\infty)}} |\widehat{f}(\alpha)|^2 < \infty,
 \end{equation*}
 where $\{e_\alpha\}_{\alpha \in \mathbb{Z}^{(\infty)}}$ denotes the canonical orthonormal basis of $L^2(\mathbb{T}^\infty)$. Furthermore, the space $\mathcal{H}^2$ is the completion of $\mathcal{P}$, the set of all Dirichlet polynomials $P(s)= \sum_{n=1}^{N}a_nn^{-s}$ with the norm defined as 
 \begin{equation*}
 	\norm{P}^2_{\mathcal{H}^2}= \lim_{T\to\infty}\frac{1}{2T}\int_{-T}^{T}|P(it)|^2 dt.
 \end{equation*} 
 %Moreover, as stated in \cite[Theorem 2]{Bayart-monatsh-2002}, the map $\mathfrak{B}: \mathcal{P}\to\mathbb{H}^2(\mathbb{T}^{\infty})$ extends to a multiplicative isometric isomorphism from $\mathcal{H}^2$ to $\mathbb{H}^2(\mathbb{T}^{\infty})$. Furthermore, $\mathbb{H}^2(\mathbb{T}^{\infty})$ is the completion of $\mathfrak{B}(\mathcal{P})$ with the norm $\norm{.}_{\mathbb{H}^2(\mathbb{T}^{\infty})}$. Bayart \cite{Bayart-monatsh-2002} has studied explicitly the space $\mathbb{H}^p(\mathbb{T}^{\infty})$ in the non-Hilbertian context. It is more natural to consider \(\mathbb{H}^2(\mathbb{D}^{\infty})\) as a function space defined on the distinguished boundary \(\mathbb{T}^\infty\). Based on the Dirichlet, Fourier and power series coefficients respectively, Hedenmalm, Lindqvist, and Seip \cite[Lemma 2.3]{Hedenmalm-Duke-1997} have stated that there are three different isometric, linear bijections among the spaces $\mathcal{H}^2$, $\mathbb{H}^2(\mathbb{T}^\infty)$, and $\mathbb{H}^2(\mathbb{D}^\infty)$. 
 
 %Let $\mathfrak{C}$ be an isometric isomorphism from $\mathbb{H}^2(\mathbb{T}^\infty)$ to $\mathbb{H}^2(\mathbb{D}^\infty)$ and $\mathfrak{L}$ be a map from $\mathcal{H}^2$ to $\mathbb{H}^2(\mathbb{D}^\infty)$ so that $\mathfrak{L}= \mathfrak{C}\mathfrak{B}$, that is, we have the following commutative isometric diagram: \\
% \begin{center}
 	%\[
 	%\begin{tikzcd}
 		%{} &  \mathbb{H}^2(\mathbb{T}^\infty) \arrow{dr}{\mathfrak{C}} \\
 		%\mathcal{H}^2 \arrow{ur}{\mathfrak{B}} \arrow{rr}{\mathfrak{L}} && \mathbb{H}^2(\mathbb{D}^\infty).
 	%\end{tikzcd}
 	%\]
%\end{center}
 
   For a holomorphic self-map $\Phi$ of the half-plane $\mathbb{C}_{1/2}$, the composition operator $C_\Phi$ on the Hardy space $\mathcal{H}^2$ is defined by $C_\Phi(f) = f \circ \Phi$. This operator initially maps $\mathcal{H}^2$ into $\operatorname{Hol}(\mathbb{C}_{1/2})$. In their seminal paper, Gordon and Hedenmalm \cite{Gordon-Michigan-1999} characterized the precise class of symbols $\Phi$ for which $C_\Phi$ acts boundedly on $\mathcal{H}^2$. Beyond classical composition, it is natural to consider operators that incorporate differentiation. To this end, given a holomorphic self-map $\Phi$ of $\mathbb{C}_{1/2}$, we define the composition-differentiation operator $D_\Phi$ by 
   \begin{equation*}
   	D_\Phi(f) = f' \circ \Phi.
   \end{equation*}
   As with $C_\Phi$, whenever $D_\Phi$ maps $\mathcal{H}^2$ into itself, the Closed Graph Theorem ensures its boundedness.
   While the differentiation operator $D$ is well-known to be unbounded on the classical Hardy space of the unit disk $\mathbb{H}^2(\mathbb{D})$, a natural problem arises: for which analytic self-maps $\Phi \colon \mathbb{D} \to \mathbb{D}$ is the composition-differentiation operator $D_\Phi$ bounded on $\mathbb{H}^2(\mathbb{D})$? Ohno \cite{Ohno-BAMS-2006} introduced $D_\Phi$ on $\mathbb{H}^2(\mathbb{D})$ in the form $C_\Phi D$ and characterized its boundedness and compactness. Subsequently, Hibschweiler and Portnoy \cite{Hibschweiler-RMJM-2005} systematically investigated the mapping properties of both $C_\Phi D$ and $D C_\Phi$. More recently, Fatehi and Hammond \cite{fatehi-PAMS-2020} provided a detailed study of composition-differentiation operators on $H^2(\mathbb{D})$. In this paper, we extend this line of inquiry to the Dirichlet-series setting, characterizing the symbols $\Phi \colon \mathbb{C}_{1/2} \to \mathbb{C}_{1/2}$ that induce bounded composition-differentiation operators $D_\Phi$ on the Hardy space $\mathcal{H}^2$.\\[0.5mm]
 
In 2003, Bayart \cite{Bayart-IJ-2003} made a seminal contribution by establishing a necessary and sufficient condition for the compactness of composition operators on $\mathcal{H}^2$ for general symbols $\Phi$, as presented in \eqref{Dipon-vasu-p1-eqn-046} below. More recently, Bayart \cite{Bayart-BLMS-2023} sharpened the sufficient condition for compactness on $\mathcal{H}^p$ for $p \in [1, +\infty)$. However, determining a necessary condition for compactness on $\mathcal{H}^2$ without imposing additional restrictions on symbols in the class $\mathcal{G}_{\ge 1}$ remains a major open problem. Addressing this gap, Kouroupis \cite{Athanasios-PAMS-2023} obtained a necessary condition for the compactness of $C_\Phi$ on $\mathcal{H}^2$ by replacing pointwise convergence with convergence in $L^1(\mathbb{T}^\infty)$. Furthermore, Queffélec and Seip \cite{Queffelec-JFA-2015} provided crucial insights into this framework by establishing decay rate estimates for the approximation numbers of compact composition operators on $\mathcal{H}^2$.
In 1999, Gordon and Hedenmalm \cite{Gordon-Michigan-1999} characterized the class of holomorphic functions $\Phi \colon \mathbb{C}_\theta \to \mathbb{C}_{1/2}$ ($\theta \in \mathbb{R}$) that generate well-defined composition operators $C_\Phi$ on the space $\mathcal{D}$ , where $\mathcal{D}$ admit functions that represents a convergent Dirichlet series in some half-plane. Identifying the analogous class of symbols $\Phi$ that induce well-defined composition-differentiation operators $D_\Phi$ on $\mathcal{D}$ constitutes one of our primary objectives. Furthermore, Gordon and Hedenmalm \cite{Gordon-Michigan-1999} determined the precise symbols $\Phi$ that generate a bounded composition operators on $\mathcal{H}^2$. They proved that the symbol $\Phi$ defined on the half-plane $\mathbb{C}_{{1}/{2}}$ generates a composition operator if, and only if,
  \begin{equation}\label{Dipon-vasu-p1-eqn-046}
  	\Phi(s) = c_0s+ \phi(s),
  \end{equation} where $c_0$ is a non-negative integer and $\phi(s)= \sum_{n=1}^{\infty}c_nn^{-s}$ converges uniformly in each half-plane $\mathbb{C}_\epsilon$, for every $\epsilon>0$, where $\phi$ satisfies the following properties:
	\begin{enumerate}
			\item[\textbf{(i)}] If $c_0=0$, then $\Phi(\mathbb{C}_0)= \phi(\mathbb{C}_0)\subset\mathbb{C}_\frac{1}{2}$.
			\item[\textbf{(ii)}] If $c_0\ge 1$, then either $\phi= 0$ or, simply $\Phi(\mathbb{C}_0)\subset\mathbb{C}_0$.
	\end{enumerate}
	The non-negative integer $c_0$ is called the characteristic of $\Phi$ or, simply \textit{char $\Phi$}. Also we denote $\mathcal{G}_0$ and $\mathcal{G}_{\ge 1}$ for the subclasses \textbf{(i)} and \textbf{(ii)} respectively. These two classes are collectively referred to as the Gordon-Hedenmalm class of symbols $\Phi$.\\[0.5mm]
	
In 2021, Brevig and Perfekt established a new compactness criterion \cite[Theorem 1.4]{Brevig-AdvMath-2021} on $\mathcal{H}^2$ for composition operators $C_\Phi$ in terms of the mean counting function \eqref{Dipon-vasu-p1-eqn-038} for symbols $\Phi \in \mathcal{G}_0$. Subsequently, Brevig and Perfekt \cite{BrevigPerfekt-JFA-2022} introduced a further compactness characterization on $\mathcal{H}^2$ for symbols $\Phi \in \mathcal{G}_{\ge 1}$ based on the Nevanlinna counting function $N_\Phi(w)$, defined for $w \in \mathbb{C}_0$ as
\begin{equation*}
	N_\Phi(w) = \sum_{s \in \Phi^{-1}(w)} \operatorname{Re}(s).
\end{equation*}
Let $\mathbb{P}$ denote the set of prime numbers and define the multiplicative sequence $\mathcal{M}(\mathbb{P}) = \{n \in \mathbb{N} \colon p \mid n \implies p \in \mathbb{P}\}$.

	\begin{defn}\cite{BrevigPerfekt-JFA-2022}
		A Dirichlet series $f$ is said to be \textit{supported} on a set of primes $\mathbb{P}$ if 
		\begin{equation*}
			f(s) = \sum_{n \in \mathcal{M}(\mathbb{P})} a_n n^{-s}.
		\end{equation*}
	\end{defn}
	
	In 2022, Brevig and Perfekt \cite{BrevigPerfekt-JFA-2022} proved that if $\Phi \in \mathcal{G}_{\ge 1}$ and $\phi$ is supported on a single prime $\mathbb{P} = \{p\}$, then the composition operator $C_\Phi$ is compact on $\mathcal{H}^2$ if and only if
	\begin{equation}\label{Dipon-vasu-p1-eqn-041}
		\lim_{\operatorname{Re}(w) \to 0^+} \frac{N_\Phi(w)}{\operatorname{Re}(w)} = 0.
	\end{equation}
	Subsequently, Kouroupis and Perfekt \cite{Athanos-JLMS-2023} established corresponding compactness criteria on the Bergman space of Dirichlet series. Motivated by the compactness condition of Brevig and Perfekt in \cite[Theorem 1.4]{Brevig-AdvMath-2021}, we aim to prove a similar decay condition for the composition-differentiation operator on $\mathcal{H}^2$ in terms of the mean counting function  $M_\Phi(w)$ is defined as, for $w \neq \Phi(+\infty)$, by
	\begin{equation}\label{Dipon-vasu-p1-eqn-038}
		M_\Phi(w) := \lim_{\sigma \to 0^+} \lim_{T \to \infty} \frac{\pi}{T} \sum_{\substack{s \in \Phi^{-1}(\{w\}) \\ |\operatorname{Im}(s)| < T \\ \sigma < \operatorname{Re}(s) < \infty}} \operatorname{Re}(s).
	\end{equation}.
	
	We focus initially on the affine single-prime symbols
	\begin{equation*}
		\Phi(s) = c_1 + c_2 2^{-s},
	\end{equation*}
	which constitute the simplest non-constant class of Gordon-Hedenmalm symbols with $c_0 = 0$. This class is particularly natural from the viewpoint of Bohr's correspondence: the Dirichlet series symbol depends on a single prime $p = 2$ and therefore reduces under the Bohr lift to an affine symbol in a single complex variable. At the same time, composition-differentiation introduces a non-trivial interaction between differentiation and the non-linear map $2^{-s}$, giving rise to an associated composition-differentiation operator on the unit disk with an exponential symbol. Consequently, this single-prime symbol provides a tractable framework within which the mapping, compactness, approximation-number, and spectral properties of composition-differentiation operators can be analyzed explicitly.\\[0.5mm]

For the affine symbol $\Phi(s) = c_1 + c_2 2^{-s}$ with $c_2 \neq 0$ and $\operatorname{Re}(c_1) \ge |c_2| + 1/2$, Muthukumar \textit{et al.} \cite{muthukumar-IEOT-2018} derived the following bounds for the operator norm of $C_\Phi$ on the Dirichlet series Hardy space $\mathcal{H}^2$:
\begin{equation*}
	\zeta(2\operatorname{Re}(c_1)) \le \|C_\Phi\|^2 \le \zeta(2\operatorname{Re}(c_1) - r|c_2|),
\end{equation*}
where $r \in (0, 1]$ is the smallest positive root of the quadratic polynomial
\begin{equation*}
	P(r) = |c_2|r^2 + (1 - 2\operatorname{Re}(c_1))r + |c_2|,
\end{equation*}
and $\zeta(s)$ denotes the Riemann zeta function for $\operatorname{Re}(s) > 1$. Furthermore, Brevig and Perfekt \cite{Brevig-JFA-2020} employed affine symbols to refine upper norm bounds for composition operators. Building on these foundations, we establish explicit norm bounds for the composition-differentiation operator $D_\Phi$ associated with the affine symbol $\Phi(s) = c_1 + c_2 2^{-s}$ under the condition $\operatorname{Re}(c_1) > |c_2| + 1/2$ with $c_2 \neq 0$.  \\[0.5mm]

In 2015, Queffélec and Seip \cite[Theorem~1.2]{Queffelec-JFA-2015} proved that if a symbol $\Phi \in \mathcal{G}_0$ induces a compact composition operator $C_\Phi$ on $\mathcal{H}^2$, then its approximation numbers $a_n(C_\Phi)$ satisfy a lower estimate 
\begin{equation}\label{Dipon-vasu-p1-eqn-020}
	a_n(C_\Phi) \gtrsim q^n\,\,\,\mbox{for some}\,\,\, 0 < q < 1.
\end{equation} 
Furthermore, they established precise upper and lower estimates for $a_n(C_\Phi)$ when $C_\Phi$ is bounded on $\mathcal{H}^2$ \cite[Theorem~4.1]{Queffelec-JFA-2015}. For a comprehensive discussions of approximation numbers and operator ideals, we refer to Section~\ref{Dipon-vasu-p1-sec-03} and Carl and Stephani \cite{Carl-Entropy-1990}. Here, our primary interest lies in obtaining precise upper and lower estimates for the approximation numbers $a_n(D_\Phi)$ associated with the affine symbol $\Phi(s) = c_1 + c_2 2^{-s}$, where $c_2 \neq 0$ and $\operatorname{Re}(c_1) > |c_2| + 1/2$.\\[0.5mm]
	 
	 The primary purpose of this paper is to establish a rigorous operator-theoretic framework for composition-differentiation operators acting on the Hardy space of Dirichlet series $\mathcal{H}^2$. In contrast to the extensively studied composition operators $C_\Phi$, the presence of differentiation introduces logarithmic weight factors in the Dirichlet coefficients, resulting in substantially different operator-theoretic behavior. We first establish characterizations of boundedness and compactness for $D_\Phi$. For the affine symbol $\Phi(s) = c_1 + c_2 2^{-s}$, we provide explicit norm estimates as well as quantitative exponential decay bounds for its approximation numbers. We further investigate the structural properties of $D_\Phi$, characterizing both self-adjointness and normality. Finally, we examine the spectral theory of compact composition-differentiation operators. In the single-prime case, the non-zero eigenvalue problem reduces to a composition-differentiation operator on the classical Hardy space of the unit disk $\mathbb{H}^2(\mathbb{D})$, thereby establishing a direct connection between operators on Dirichlet series and analytic operators on the disk. Overall, these results yield a unified operator-theoretic treatment of composition-differentiation operators on $\mathcal{H}^2$. \\[0.5mm]

	 	  The paper is organized as follows. At the beginning of Section~\ref{Dipon-vasu-p1-sec-02}, we characterize the structure of symbols $\Phi$ mapping $\mathbb{C}_{1/2}$ into itself that induce well-defined composition-differentiation operators $D_\Phi$ on $\mathcal{H}^2$. We establish a decay condition for the mean counting function $M_\Phi$ near the boundary of the half-plane, which characterizes compactness of $D_\Phi$ on $\mathcal{H}^2$. Furthermore, via a comparison-type principle \cite{Athanosios-Revista-2024}, we demonstrate that this boundary decay is equivalent to a corresponding decay condition in terms of the Green's function, which we also analyze through harmonic measure. Additionally, we derive explicit mapping properties of $\Phi$ that generate bounded composition-differentiation operators $D_\Phi$ and obtain precise norm estimates for $D_\Phi$ when $\Phi$ is an affine symbol $c_1 + c_2 2^{-s}$ with $c_1, c_2 \in \mathbb{C}$, $c_2 \neq 0$, and satisfying the sharp boundedness condition $\operatorname{Re}(c_1) > |c_2| + 1/2$. In Section~\ref{Dipon-vasu-p1-sec-03}, inspired by the work of Queffélec and Seip \cite{Queffelec-JFA-2015}, we establish upper and lower bounds for the approximation numbers of $D_\Phi$ on $\mathcal{H}^2$. Finally, Section~\ref{Dipon-vasu-p1-sec-04} explores spectral and operator-theoretic properties of $D_\Phi$ for symbols in $\mathcal{G}_0$. \\[0.5mm]
	 	   
	 	   	\paragraph{\textbf{Notation}} Throughout this article, $C$ denotes a positive constant which may vary from line to line. When necessary, we write that $C=C(x)$ to indicate that the constant depends on a parameter $x.$ If $f$ and $g$ are two real functions defined on the same set $S$, the notation $f \asymp g$ means that there exists a constant $C\ge 1$ such that $C^{-1}f(x) \le g(x) \le Cf(x)$, for all $x$, while $f(x)\lesssim g(x)$ signifies that there exists a constant $C>0$ such that $f(x)\le Cg(x)$, for all $x$. Finally, for $w\in\mathbb{C}$ and $r>0$, we denote by  $\mathbb{D}(w,r):=\{\xi\in\mathbb{C}: |\xi-w|<r\}$ the open disk centered at $w$ and radius $r$.
	 	   	
\section{Compactness and boundedness of the operator $D_\Phi$}\label{Dipon-vasu-p1-sec-02}

Here we study composition-differentiation operators $D_\Phi$ acting on the Hardy space of Dirichlet series $\mathcal{H}^2$. First, we establish the structural requirements on the symbol $\Phi$ that generate a well-defined composition-differentiation operator on $\mathcal{D}$. We then characterize the precise mapping properties of $\Phi$ that induce bounded and compact operators $D_\Phi$ on $\mathcal{H}^2$.

\begin{thm}\label{Dipon-vasu-p1-thm-01}
	The composition-differentiation operator $D_\Phi$ generated by a holomorphic symbol $\Phi \colon \mathbb{C}_\theta \to \mathbb{C}_{1/2}$, for some $\theta \in \mathbb{R}$, acts boundedly on the Dirichlet series space $\mathcal{D}$ if and only if $\Phi$ is of the form
	\begin{equation}\label{Dipon-vasu-p1-eqn-022}
		\Phi(s) = c_0 s + \phi(s),
	\end{equation}
	where $c_0 \in \mathbb{N} \cup \{0\}$ and $\phi \in \mathcal{D}$ is a convergent Dirichlet series.
\end{thm}

\begin{proof}[Sketch of the proof]
	The proof of Theorem~\ref{Dipon-vasu-p1-thm-01} builds upon the strategy developed by Gordon and Hedenmalm \cite[Theorem A]{Gordon-Michigan-1999} for composition operators $C_\Phi$ on $\mathcal{H}^2$. For completeness, we outline the primary steps. Crucially, a Dirichlet series $f$ and its derivative $f'$ share the same abscissa of absolute convergence, $\sigma_a(f) = \sigma_a(f')$.
	
	To establish necessity, assuming $f' \circ \Phi \in \mathcal{D}$ for every $f \in \mathcal{D}$, we must verify the structural form of $\Phi$ given in \eqref{Dipon-vasu-p1-eqn-022}. This requires demonstrating that $\Phi(s)/s \to c_0$ as $\operatorname{Re}(s) \to +\infty$ and that $\Phi \in \mathcal{D}$, which implies that $\Phi(s)$ is a convergent Dirichlet series generated by a multiplicative semigroup $\{1 + j/k\}_{j \ge 1}$ consisting exclusively of positive integers.
	Conversely, for sufficiency, we assume $\Phi$ has the specified structure with $\operatorname{Re}(c_1) > 1/2$ and prove that $f' \circ \Phi \in \mathcal{D}$ for all $f \in \mathcal{D}$ by showing that the Dirichlet series $f' \circ \Phi$ converges in a suitable right half-plane. This completes the proof of Theorem~\ref{Dipon-vasu-p1-thm-01}.
\end{proof}

\subsection{Compactness of the operator $D_\Phi$}
We focus on the class $\mathcal{G}_0$, which corresponds to the case of zero characteristic for the symbol $\Phi$ of the operator $D_\Phi$ on $\mathcal{H}^2$. From Carlson's lemma \cite[Lemma 3.2]{Hedenmalm-Duke-1997}, we deduce a Littlewood-Paley type formula \cite[Chapter 10]{Shapiro-Book-1992} for the norm in $\mathcal{H}^2$:
\begin{equation}\label{Dipon-vasu-p1-eqn-003}
	\|f\|_{\mathcal{H}^2}^2 = |f(+\infty)|^2 + 2 \lim_{\sigma_0 \to 0^+} \lim_{T \to \infty} \frac{1}{T} \int_{\sigma_0}^{\infty} \int_{-T}^{T} \sigma |f'(\sigma + it)|^2 \, dt \, d\sigma,
\end{equation}
valid for every $f \in \mathcal{H}^2$ satisfying $\sigma_u(f) \le 0$, where
\begin{equation*}
	\sigma_u(f) := \inf \left\{ \sigma \in \mathbb{R} \colon \sum_{n=1}^{\infty} a_n n^{-s} \text{ converges uniformly in } \overline{\mathbb{C}_\sigma} \right\}
\end{equation*}
denotes the abscissa of uniform convergence of $f$. Analogously, one defines the abscissa of ordinary convergence $\sigma_c(f)$ and the abscissa of absolute convergence $\sigma_a(f)$.

For a symbol $\Phi \in \mathcal{G}_0$, applying a suitable non-injectiver change of variables to identity \eqref{Dipon-vasu-p1-eqn-003} yields
\begin{equation}\label{Dipon-vasu-p1-eqn-005}
	\|D_\Phi f\|_{\mathcal{H}^2}^2 = \|f' \circ \Phi\|_{\mathcal{H}^2}^2 = |f'(\Phi(+\infty))|^2 + \frac{2}{\pi} \int_{\mathbb{C}_{1/2}} |f''(w)|^2 M_\Phi(w) \, dA(w)
\end{equation}
for all $f \in \mathcal{H}^2$, where the $M_\Phi(w)$ denotes mean counting function defined in \eqref{Dipon-vasu-p1-eqn-038}.
Furthermore, the mean counting function satisfies a Littlewood-type inequality of the following form.

\begin{thm}\cite{Brevig-AdvMath-2021}\label{Dipon-vasu-p1-thm-010}
	Let $\Phi$ be a Dirichlet series symbol with $\operatorname{char}(\Phi) = 0$. Then the mean counting function $M_\Phi(w)$ exists for every $w \in \mathbb{C}_{1/2} \setminus \{\Phi(+\infty)\}$, and satisfies the point-wise estimate
	\begin{equation*}
		M_\Phi(w) \le \log \left| \frac{\overline{w} + \Phi(+\infty)-1}{w - \Phi(+\infty)} \right|.
	\end{equation*}
\end{thm}

The change of variables formula \eqref{Dipon-vasu-p1-eqn-005} is the analogue of \cite[Theorem 1.3]{Brevig-AdvMath-2021} established for composition operators $C_\Phi$ on $\mathcal{H}^2$. Furthermore, in view of \cite[Theorem 1.2]{Brevig-AdvMath-2021}, the mean counting function is given explicitly by
\begin{equation*}
	M_\Phi(w) = \log \left| \frac{w + \bar{v} - 1}{w - v} \right|
\end{equation*}
for quasi-every $w = \sigma + i t \in \mathbb{C}_{1/2}$ with $w \neq v$, where $v := \Phi(+\infty) \in \mathbb{C}_{1/2}$. \\[0.5mm] 
 
The characterization of boundedness and compactness for the composition-differentiation operator $D_\Phi$ in terms of the mean counting function $M_\Phi$ relies on estimates for the reproducing kernels in $\mathcal{H}^2$. For any $a \in \mathbb{C}_{1/2}$, the reproducing kernel $K_a$ associated with point evaluation of functions $f \in \mathcal{H}^2$ satisfies
\begin{equation*}
	\langle f, K_a \rangle = f(a),
\end{equation*}
and is explicitly given by
\begin{equation*}
	K_a(s) = \zeta(s + \bar{a}), \qquad s \in \mathbb{C}_{1/2},
\end{equation*}
where $\zeta$ denotes the Riemann zeta function. Similarly, we denote by $K_a^{(1)}$ the reproducing kernel for point evaluation of the derivative, satisfying
\begin{equation*}
	\langle f, K_a^{(1)} \rangle = f'(a),
\end{equation*}
where
\begin{equation*}
	K_a^{(1)}(s) = -\sum_{n=1}^{\infty} (\log n) n^{-(s + \bar{a})}.
\end{equation*}
Direct calculation then yields the exact norms
\begin{equation*}
	\|K_a\|_{\mathcal{H}^2} = \sqrt{\sum_{n=1}^{\infty} n^{-(a + \bar{a})}} = \sqrt{\zeta(2\operatorname{Re}(a))}
\end{equation*}
and
\begin{equation*}
	\|K_a^{(1)}\|_{\mathcal{H}^2} = \sqrt{\sum_{n=1}^{\infty} (\log n)^2 n^{-(a + \bar{a})}} = \sqrt{\zeta''(2\operatorname{Re}(a))},
\end{equation*}
for every $a \in \mathbb{C}_{1/2}$. We are now in a position to state and prove one of the principal results of this paper.
  
\begin{thm}\label{Dipon-vasu-p1-thm-09}
		Let $\Phi\in\mathcal{G}_0$ and $M_\Phi$ denote the associated mean counting function \eqref{Dipon-vasu-p1-eqn-038}. Then the induced composition-differentiation operator $D_\Phi$ is compact on $\mathcal{H}^2$ if, and only if,
		\begin{equation}\label{Dipon-vasu-p1-eqn-039}
			\lim_{\real(w)\to\frac{1}{2}^+} \frac{M_\Phi(w)}{(\real(w) - \frac{1}{2})^3} = 0.
		\end{equation}
	\end{thm}

By applying the reproducing kernels of $\mathcal{H}^2$ together with the sub-mean-value property of the mean counting function (as stated below in Lemma~\ref{Dipon-vasu-p1-lem-004}), it is easy to demonstrate that condition \eqref{Dipon-vasu-p1-eqn-039} is necessary for the compactness of $D_\Phi$. Meanwhile, the change of variables formula \eqref{Dipon-vasu-p1-eqn-005} plays a fundamental role in establishing sufficiency in Theorem~\ref{Dipon-vasu-p1-thm-09}. The following remark provides an essential analogue to \cite[Theorem 1.4]{Athanos-JLMS-2023} within the setting of Theorem~\ref{Dipon-vasu-p1-thm-09}.

 \begin{rem}\label{Dipon-vasu-p1-rem-01}
 		Consider the symbol $\Phi\in\mathcal{G}_0$. If the composition-differentiation operator $D_\Phi$ is bounded on $\mathcal{H}^2$, then for every $\delta>0$, there exists a constant $C$ that depends on $\delta$ such that
 		\begin{equation}\label{Dipon-vasu-p1-eqn-040}
 			\frac{M_\Phi(w)}{(\real(w) - \frac{1}{2})^3}< C,
 		\end{equation}
 		for every $w\in\mathbb{C}_{{1}/{2}}$ that does not attain the open disk $\mathbb{D}\left((\Phi(+\infty)), \delta\right)$.
 \end{rem}
 
We note that Remark~\ref{Dipon-vasu-p1-rem-01} is a direct analogue of \cite[Corollary 3.2]{Ohno-BAMS-2006} for the classical Hardy space $\mathbb{H}^2(\mathbb{D})$, particularly regarding its connection to the Nevanlinna counting function \cite{Shapiro-annals-1987}.
 
 \begin{defn}
	A Dirichlet series $f$ with uniform abscissa $\sigma_u(f)\le 0$ is defined in the \textit{Nevanlinna class of Dirichlet series} $\mathcal{N}_u$ if the following condition holds:
\begin{equation*}
	{\limsup_{\sigma\to 0^+}}\lim_{T\to\infty}\frac{1}{2T}\int_{-T}^{T}\log^+|f(\sigma +it)|\,dt<\infty,
\end{equation*}
where $\log^+(x) = \max\{\log(x), 0\}$.
 \end{defn}
The proof of Theorem~\ref{Dipon-vasu-p1-thm-09} relies on the following two fundamental tools: the sub-mean-value property and the meromorphic extension of the Riemann zeta function.
 	\begin{lem}\cite{Brevig-AdvMath-2021}\label{Dipon-vasu-p1-lem-004}
 		Let $f\in\mathcal{N}_u$. Then the mean counting function $M_f$ satisfies the following submean value property:
 		\begin{equation*}
 			M_f(w)\le\frac{1}{\pi r^2}\int_{\mathbb{D}(w,r)} M_f(s)\, ds,
 		\end{equation*}
 		for every open disk $\mathbb{D}(w,r)$ which does not attain $f(+\infty)$.
 	\end{lem}
 
 	\begin{lem}\label{Dipon-vasu-p1-lem-005}\cite{Athanos-JLMS-2023}
 		For $\real(s)>1$, assume $\zeta(s) = \sum_{n=1}^{\infty}n^{-s}$. Then, on the half-plane $\mathbb{C}_0$, there exists an entire function $E$ such that
 		\begin{equation*}
 			\zeta(s) = \frac{1}{s-1} + E(s),\,\,\, s\in\mathbb{C}_1.
 		\end{equation*}
 	\end{lem}
 	
The following two lemmas are key in proving necessity for the compactness of the operator $D_\Phi$ via the mean counting function $M_\Phi$. As their proofs closely follow the arguments in \cite{Athanos-JLMS-2023}, we omit the details.
 
\begin{lem}\label{Dipon-vasu-p1-lem-006}
		Let $\Phi\in\mathcal{G}_0$ and $\delta\in(0,1)$ be fixed. Then for every $\varepsilon>0$, there is some $1/2<\theta<(1/2+\real(\Phi(+\infty)))/2$, such that if $1/2<\real(w)<\theta$, then
		\begin{equation*}
			M_\Phi(w)\le\varepsilon^2\frac{(\real(w)-1/2)^3}{|w-\Phi(+\infty)|^{1+\delta}}.
		\end{equation*}
	\end{lem}
	
  The above Lemma~\ref{Dipon-vasu-p1-lem-006} is an immediate consequence of \cite[Proposition 5.4]{Athanos-JLMS-2023} and our compactness characterization \eqref{Dipon-vasu-p1-eqn-039} for $D_\Phi$.

	\begin{lem}\label{Dipon-vasu-p1-lem-007}
		Let $\Phi\in\mathcal{G}_0$ and $\{f_n\}_{n\ge 1}$ be a weakly convergent sequence in $\mathcal{H}^2$ that converges weakly to $0$. Then, for each $\theta>1/2$, we have
		\begin{equation*}
			\lim_{n\to\infty}\bigg(|f'_n(\Phi(\infty))|^2+\frac{2}{\pi}\int_{\real(w)\ge\theta}|f''_n(w)|^2M_\Phi(w)\, dA(w)\bigg) = 0.
		\end{equation*}
	\end{lem}
In view of \cite[Lemma 5.6]{Athanos-JLMS-2023}, Lemma~\ref{Dipon-vasu-p1-lem-007} follows clearly. Moreover, the argument relies on the fact that
\begin{equation*}
	\lim_{n \to \infty} f'_n(\Phi(+\infty)) = \lim_{n \to \infty} \langle f_n, K_{\Phi(+\infty)}^{(1)} \rangle = 0,
\end{equation*}
since $f_n \xrightarrow{w} 0$ in $\mathcal{H}^2$. The remainder of the proof proceeds by invoking \cite[Lemma 18]{Bayart-monatsh-2002}. 
The following lemma plays a key role in the proof of Theorem~\ref{Dipon-vasu-p1-thm-09}.
\begin{lem}\label{Dipon-vasu-p1-lem-008}
		Let $v\in\mathbb{C}_{{1}/{2}}$ and $\delta>0$. Then, there exists a constant $C$ depends on $\delta$ such that
		\begin{equation*}
			\int_{1/2<\real(w)<\theta}|f''(w)|^2\frac{(\real(w)-1/2)^3}{|w-v|^{1+\delta}}\,dw\le\frac{C}{(\real(v)-\theta)^{1+\delta}}\norm{f}_{\mathcal{H}^2}^2,
		\end{equation*}
		for every $f\in\mathcal{H}^2$ and $1/2<\theta<\real(v)$.
	\end{lem}
	\begin{proof}
		Since $1/2<\real(w)<\theta$ and the $\mathcal{H}^2$ norm remains invariant under vertical translation, we can assume without loss of generality that \( v \) is real. Consider the function
		\begin{equation*}
			w(t) = \frac{1}{|w-v|^{1+\delta}} = \frac{1}{|\sigma+it-v|^{1+\delta}}\le\frac{1}{|\theta+it-v|^{1+\delta}}.
		\end{equation*}
		By applying \cite[Lemma 7.2]{Brevig-AdvMath-2021} with the decreasing function $w(t)$ along the vertical line $1/2<\real(w)=\sigma<\theta$ to the Dirichlet series $f''(\sigma-1/2+s)$, we obtain
		\begin{equation*}
			\int_{-\infty}^{+\infty}\frac{|f''(\sigma+it)|^2}{|\sigma+it-v|^{1+\delta}}\,dt\le\frac{1}{(\real(v)-\theta)^{1+\delta}}\sum_{n\ge 1}|a_n|^2\frac{(\log n)^4}{n^{2(\sigma-1/2)}}.
		\end{equation*}
		Thus, we have
		\begin{align*}
			\int_{1/2<\real(w)<\theta}|f''(w)|^2\frac{(\real(w)-1/2)^3}{|w-v|^{1+\delta}}\,dw &\le\frac{\sum_{n\ge 1}|a_n|^2}{(\real(v)-\theta)^{1+\delta}}\int_{1/2}^{\theta}\frac{(\log n)^4}{n^{2(\sigma-1/2)}}(\sigma-1/2)^3\,d\sigma\\&\le\frac{\norm{f}_{\mathcal{H}^2}^2}{(\real(v)-\theta)^{1+\delta}}\int_{1/2}^{\infty}\frac{(\log n)^4}{n^{2(\sigma-1/2)}}(\sigma-1/2)^3\,d\sigma\\&\le\frac{3}{8}\frac{1}{(\real(v)-\theta)^{1+\delta}}\norm{f}_{\mathcal{H}^2}^2.
		\end{align*}
		This completes the proof.
\end{proof}

	\begin{pf}[\bf{Proof of Theorem \ref{Dipon-vasu-p1-thm-09}}]
		Let the operator $D_\Phi$ be compact. Our goal is to prove the condition \eqref{Dipon-vasu-p1-eqn-039}. Consider an arbitrary sequence $\{w_j\}_{j\ge 1}$ in the half-plane $\mathbb{C}_{{1}/{2}}$ such that $\real(w_j)$ converges to  $1/2^+$ as $j\to\infty$. Without loss of generality, we can assume that
			\begin{equation}\label{Dipon-vasu-p1-eqn-041}
				\real(w_j)\le\frac{1/2 + \real(\Phi(\infty))}{2}.
			\end{equation}
			Indeed, setting $v := \Phi(+\infty)$, estimate \eqref{Dipon-vasu-p1-eqn-041} yields
			\begin{equation*}
				\real w_j - \frac{1}{2} \le \frac{\real v - \frac{1}{2}}{2}.
			\end{equation*}
			For any $s \in \mathbb{D}(w_j, r_j)$, where $r_j = \frac{1}{2}\bigl(\real w_j - \frac{1}{2}\bigr)$, we have
			\begin{equation*}
				\real s \le \real w_j + r_j = \frac{1}{2} + \frac{3}{2}\left(\real w_j - \frac{1}{2}\right) \le \frac{1}{2} + \frac{3}{4}\left(\real v - \frac{1}{2}\right).
			\end{equation*}
			Consequently,
			\begin{equation*}
				\real v - \real s \ge \frac{1}{4}\left(\real v - \frac{1}{2}\right) > 0.
			\end{equation*}
			It follows directly that
			\begin{equation*}
				\operatorname{dist}\bigl(\mathbb{D}(w_j, r_j), \Phi(+\infty)\bigr) \ge \frac{1}{4}\left(\real (\Phi(+\infty)) - \frac{1}{2}\right).
			\end{equation*}
			This ensures that $\Phi(+\infty)$ is uniformly bounded on the open disk $\mathbb{D}(w_j,r_j)\subset\mathbb{C}_{{1}/{2}}\setminus \{\Phi(+\infty)\}$, where $r_j = {(\real(w_j)-1/2)}/{2}$. Hence from Lemma \ref{Dipon-vasu-p1-lem-004}, we have
			\begin{equation}\label{Dipon-vasu-p1-eqn-042}
				\frac{M_\Phi(w_j)}{(\real(w_j)-1/2)^3}\le\frac{1}{8\pi r_j^5}\int_{\mathbb{D}(w_j,r_j)}M_\Phi(s)ds.
			\end{equation}
			For $w\in\mathbb{C}_{{1}/{2}}$, consider the function
			$$K_w(s) = \frac{\zeta(s+\overline{w})}{\sqrt{\zeta(2\real(w))}},$$ which serves as a normalized reproducing kernel Hilbert space on $\mathcal{H}^2$. Clearly, the sequence $\{K_{w_j}\}_{j\ge 1}$ converges weakly to $0$. By the compactness of the operator $D_\Phi$, we have  $D_\Phi(K_{w_j})\to 0$ as $j\to\infty$ on $\mathcal{H}^2$. 
			After taking second derivatives of the normalized kernel functions yields
			\begin{equation*}
				K_{w_j}''(s) = \frac{\zeta''(s + \overline{w_j})}{\sqrt{\zeta(2\real w_j)}}.
			\end{equation*}
			By the virtue of Lemma \ref{Dipon-vasu-p1-lem-005}, the Riemann zeta function satisfies $\zeta(z) = (z-1)^{-1} + E(z)$, where $E$ is an entire function, so that
			\begin{equation*}
				\zeta''(z) = \frac{2}{(z-1)^3} + E''(z).
			\end{equation*}
			For $s \in \mathbb{D}(w_j, r_j)$ with radius $r_j := \frac{1}{2}\bigl(\real w_j - \frac{1}{2}\bigr)$, we have $|s + \overline{w_j} - 1| \asymp r_j$. Since $E''$ is bounded on compact subsets of $\mathbb{C}$, it follows that for all sufficiently large $j$,
			\begin{equation*}
				|\zeta''(s + \overline{w_j})| \gtrsim r_j^{-3}.
			\end{equation*}
			Moreover, near the pole $z = 1$ we have $\zeta(2\real w_j) = \zeta(1 + 4r_j) \asymp r_j^{-1}$. Consequently,
			\begin{equation}\label{Dipon-vasu-p1-eqn-004}
				|K_{w_j}''(s)|^2 \gtrsim \frac{r_j^{-6}}{r_j^{-1}} = r_j^{-5}, \qquad s \in \mathbb{D}(w_j, r_j),
			\end{equation}
			where the implied constant is independent of both $j$ and $s$.
			Applying the estimates \eqref{Dipon-vasu-p1-eqn-042} and \eqref{Dipon-vasu-p1-eqn-004} and extending the domain of integration to the half-plane $\mathbb{C}_{1/2}$, we obtain from \eqref{Dipon-vasu-p1-eqn-005} that
			\begin{align*}
				\frac{M_\Phi(w_j)}{(\real(w_j)-1/2)^3} &\le\frac{1}{8\pi r_j^5}\int_{\mathbb{C}_{{1}/{2}}}M_\Phi(s)\,ds\\&\lesssim\frac{1}{8\pi C}\int_{\mathbb{C}_{{1}/{2}}}|K''_{w_j}(s)|^2M_\Phi(s)\,ds\\&\le\frac{1}{16C}\norm{D_\Phi(K_{w_j})}^2_{\mathcal{H}^2}\to 0\,\,\,\mbox{as}\,\, j\to\infty.
			\end{align*}
			Since $\{w_j\}_{j\ge 1}$ is an arbitrary sequence in $\mathbb{C}_{{1}/{2}}$, we conclude that 
			\begin{equation*}
				\lim_{\real(w)\to\frac{1}{2}^+} \frac{M_\Phi(w)}{(\real(w) - \frac{1}{2})^3} = 0.
			\end{equation*}
			
Conversely, suppose that condition \eqref{Dipon-vasu-p1-eqn-039} holds. By Lemma~\ref{Dipon-vasu-p1-lem-006}, for every $\varepsilon > 0$, there exists a corresponding choice of parameters such that
\begin{equation*}
	M_\Phi(w) \le \varepsilon^2 \frac{(\operatorname{Re}(w) - 1/2)^3}{|w - \Phi(+\infty)|^{1+\delta}}.
\end{equation*}
To establish that $D_\Phi$ is compact on $\mathcal{H}^2$, let $\{f_n\}_{n \ge 1}$ be an arbitrary sequence in $\mathcal{H}^2$ that is bounded in norm and converges weakly to $0$. We must show that $\|D_\Phi f_n\|_{\mathcal{H}^2}^2 \to 0$ as $n \to \infty$. 
By the change of variables formula \eqref{Dipon-vasu-p1-eqn-005}, for any fixed $\theta > 1/2$, we decompose the norm as
\begin{align*}
	\|D_\Phi f_n\|_{\mathcal{H}^2}^2 &= \|f_n' \circ \Phi\|_{\mathcal{H}^2}^2 \\
	&= |f_n'(\Phi(+\infty))|^2 + \frac{2}{\pi} \int_{\mathbb{C}_{1/2}} |f_n''(w)|^2 M_\Phi(w) \, dA(w) \\
	&= |f_n'(\Phi(+\infty))|^2 + \frac{2}{\pi} \left( \int_{\operatorname{Re}(w) \ge \theta} + \int_{1/2 < \operatorname{Re}(w) < \theta} \right) |f_n''(w)|^2 M_\Phi(w) \, dA(w).
\end{align*}
Applying Lemma~\ref{Dipon-vasu-p1-lem-007}, the weak convergence $f_n \xrightarrow{w} 0$ ensures that
\begin{equation}\label{Dipon-vasu-p1-eqn-043}
	\lim_{n \to \infty} \left( |f_n'(\Phi(+\infty))|^2 + \frac{2}{\pi} \int_{\operatorname{Re}(w) \ge \theta} |f_n''(w)|^2 M_\Phi(w) \, dA(w) \right) = 0.
\end{equation}
Furthermore, invoking Lemma~\ref{Dipon-vasu-p1-lem-006} together with Lemma~\ref{Dipon-vasu-p1-lem-008} for $v = \Phi(+\infty)$ and $1/2 < \theta \le \frac{1/2 + \operatorname{Re}(\Phi(+\infty))}{2} < \operatorname{Re}(v)$, we obtain the uniform bound
\begin{align}\label{Dipon-vasu-p1-eqn-044}
	\int_{1/2 < \operatorname{Re}(w) < \theta} |f_n''(w)|^2 M_\Phi(w) \, dA(w) 
	&\le \varepsilon^2 \int_{1/2 < \operatorname{Re}(w) < \theta} |f_n''(w)|^2 \frac{(\operatorname{Re}(w) - 1/2)^3}{|w - \Phi(+\infty)|^{1+\delta}} \, dA(w) \nonumber \\
	&\le C \varepsilon^2 \|f_n\|_{\mathcal{H}^2}^2,
\end{align}
where $C > 0$ is an absolute constant. Combining \eqref{Dipon-vasu-p1-eqn-043} and \eqref{Dipon-vasu-p1-eqn-044}, and taking $n \to \infty$ followed by $\varepsilon \to 0$, we conclude that $\norm{D_\Phi f_n}_{\mathcal{H}^2}^2 \to 0$ as $n\to\infty$.
Thus, $D_\Phi$ maps weakly convergent sequences to strongly convergent ones, proving that $D_\Phi$ is compact on $\mathcal{H}^2$.
\end{pf}

 The following remark serves as a geometric characterization on the symbol \( \Phi \) that induces a compact composition-differentiation operator on \( \mathcal{H}^2 \).
	\begin{rem}
		Assume the symbol $\Phi\in\mathcal{G}_0$. If the range of the half-plane $\mathbb{C}_0$ is contained in the angular sector $\{s\in\mathbb{C}_{\frac{1}{2}}: |\arg(s-1/2)|<\pi/2\alpha\}$, for some $\alpha>3$, then the operator $D_\Phi$ is compact.
		Moreover, if there exists $\delta > 0$ such that
		\begin{equation*}
			\Phi(\mathbb{C}_0) \cap \left\{ \frac{1}{2} < \real s < \frac{1}{2} + \delta \right\} \subset \bigcup_{j=1}^N \Omega_j,
		\end{equation*}
		where each angular sector $\Omega_j$ is given by
		\begin{equation*}
			\Omega_j = \left\{ s \in \mathbb{C}_{1/2} : \left| \operatorname{arg}\left(s - \frac{1}{2} - i\tau_j\right) \right| < \beta_j \right\}
		\end{equation*}
		with $0 < \beta_j < \pi/6$ for $j = 1, \dots, N$, then $D_\Phi$ is compact on $\mathcal{H}^2$.
		The key part of this proof of such condition involves following Green's function 
		\begin{equation*}
			g_{\Omega_{\alpha}}(z, w) = \log \left| \frac{\left(z - \frac{1}{2}\right)^{\alpha} + \left(w - \frac{1}{2}\right)^{\alpha}}{\left(z - \frac{1}{2}\right)^{\alpha} - \left(w - \frac{1}{2}\right)^{\alpha}} \right|,\,\,\,\mbox{for}\,\,\,\alpha>3
		\end{equation*}
		of the potential range of the half-plane that satisfies the comparison principle \cite[Sec. 4]{Athanosios-Revista-2024} 
		\begin{equation}\label{Dipon-vasu-p1-eqn-049}
			M_\Phi(w) \lesssim g_{\Omega_\alpha}(w,v),
		\end{equation}
		 where $v=\Phi(+\infty)$ uniformly distinct from $w$.
	\end{rem}
	
 For a detailed background on Green's functions and their fundamental properties, we refer to \cite[Section~4.4]{Ransford-Potential-1995}.
   Moreover, we obtain the following characterization of the compactness of the operator $D_\Phi$ in terms of the decay rate of the Green's function $g_D$ near the boundary $\partial\mathbb{C}_{1/2}$.
	\begin{cor}\label{Dipon-vasu-p1-cor-01}
        	Let $D \subset \mathbb{C}_{1/2}$ be a Greenian domain such that $\Phi(\mathbb{C}_0) \subset D$ and $v := \Phi(+\infty) \in D$. If the Green function $g_D(\cdot, v)$ of $D$ with pole at $v$ satisfies
			\begin{equation}\label{Dipon-vasu-p1-eqn-051}
				g_D(w, v) = o\left(\left(\real w - \frac{1}{2}\right)^3\right), \qquad \text{as }\quad \real w \to \frac{1}{2}^+,
			\end{equation}
			then the operator $D_\Phi$ is compact on $\mathcal{H}^2$.
	\end{cor}
	The proof of Corollary \ref{Dipon-vasu-p1-cor-01} essentially follows from the comparison principle \eqref{Dipon-vasu-p1-eqn-049}. 
	The compactness of $D_\Phi$ then follows directly from Theorem~\ref{Dipon-vasu-p1-thm-09}. Now we will relate the condition of  Corollary \ref{Dipon-vasu-p1-cor-01} in terms of the harmonic measure through the following remark. 
	\begin{rem}
	Let $D \subset \mathbb{C}_{1/2}$ be a bounded Greenian domain, and let $\omega_D(z, E)$ denote the harmonic measure of a Borel set $E \subset \partial D$ evaluated at $z \in D$. Let $U$ be a neighborhood of the boundary point $1/2$ such that $v := \Phi(+\infty) \notin \overline{U}$, and define the relative subdomain $V := D \cap U.$
	The Green function $g_D(\cdot, v)$ is positive and harmonic on $V$. Under standard regularity assumptions on $\partial D$, $g_D(\cdot, v)$ vanishes continuously on $\partial D \cap U$. Consequently, by the harmonic-measure integral representation, we have
	\begin{equation}\label{Dipon-vasu-p1-eqn-050}
		g_D(w, v) = \int_{\partial V \cap D} g_D(\zeta, v) \, d\omega_V(w, \zeta), \qquad w \in V.
	\end{equation}
	Since $v \notin \overline{U}$, the integrand $g_D(\zeta, v)$ remains uniformly bounded on the relative interior boundary $\partial V \cap D$. Hence, there exists a positive constant $C$ such that
	\begin{equation*}
		g_D(v,w) \le C\,\omega_V(w,\partial U \cap D).
	\end{equation*}
	Thus if 
	\begin{equation*}
		\omega_V(w,\partial U \cap D) = o\left(\left(\real w - \frac{1}{2}\right)^3\right), \qquad \text{as }\quad \real w \to \frac{1}{2}^+,
	\end{equation*}
	then the desired Green's function decay \eqref{Dipon-vasu-p1-eqn-051} follows locally.
	\end{rem}	
	
For a comprehensive treatment of harmonic measure and its potential-theoretic properties, we refer to \cite[Section~4.3]{Ransford-Potential-1995}.

\subsection{Boundedness of the operator $D_\Phi$}

The following counterexample demonstrates that the mapping condition $\Phi(\mathbb{C}_0) \subset \mathbb{C}_{1/2}$ alone is not sufficient to guarantee the boundedness of the operator $D_\Phi$ on the space $\mathcal{H}^2$.

\begin{example}
	Fix $v := \Phi(+\infty) \in \mathbb{C}_{1/2}$ and consider the symbol $\Phi_v : \mathbb{C}_0 \to \mathbb{C}_{1/2}$ defined by
	\begin{equation}\label{Dipon-vasu-p1-eqn-007}
		\Phi_v(s) = \frac{(\overline{v}-1)2^{-s} + v}{1 - 2^{-s}}.
	\end{equation}
	For $\delta > 0$ sufficiently small, set $a_\delta := \frac{1}{2} + \delta$ and let $K_{a_\delta}$ denote the normalized reproducing kernel at $a_\delta$, so that $\|K_{a_\delta}\|_{\mathcal{H}^2} = 1$. By the kernel derivative estimates \eqref{Dipon-vasu-p1-eqn-004}, there exists a constant $c > 0$ such that
	\begin{equation*}
		|K_{a_\delta}''(w)|^2 \gtrsim \frac{1}{\delta^5} \qquad \text{for all } w \in \mathbb{D}(a_\delta, c\delta).
	\end{equation*}
	On the disk $\mathbb{D}(a_\delta, c\delta)$, we have $\real w - \frac{1}{2} \asymp \delta$. Moreover, the explicit evaluation of $M_{\Phi_v}$ yields $M_{\Phi_v}(w) \gtrsim \delta$ uniformly on $\mathbb{D}(a_\delta, c\delta)$. Applying Littlewood-Paley identity, we obtain
	\begin{equation*}
		\begin{aligned}
			\|D_{\Phi_v} K_{a_\delta}\|_{\mathcal{H}^2}^2 &\ge \frac{2}{\pi} \int_{\mathbb{D}(a_\delta, c\delta)} |K_{a_\delta}''(w)|^2 M_{\Phi_v}(w) \, dA(w) \\
			&\gtrsim \delta^{-5} \cdot \delta \cdot \operatorname{Area}\bigl(\mathbb{D}(a_\delta, c\delta)\bigr) \\
			&\asymp \delta^{-5} \cdot \delta \cdot \delta^2 = \delta^{-2}.
		\end{aligned}
	\end{equation*}
	Consequently, $\|D_{\Phi_v} K_{a_\delta}\|_{\mathcal{H}^2} \gtrsim \delta^{-1} \to \infty$ as $\delta \to 0^+$, establishing that $D_{\Phi_v}$ is unbounded on $\mathcal{H}^2$.
\end{example}

%In view of  \cite[Theorem 1.2]{Brevig-AdvMath-2021} and \cite[Lemma 2.3]{Brevig-AdvMath-2021}, we observe that for sufficiently small $\epsilon>0$, a simple computation based on  \eqref{Dipon-vasu-p1-eqn-005} gives
%\begin{align}\label{Dipon-vasu-p1-eqn-020}
%\norm{D_\Phi(f)}_{\mathcal{H}^2}^2&\ge\int_{\mathbb{R}}\int_{\frac{1}{2}}^{\frac{1}{2}+\epsilon}|f''(\sigma+it)|^2\log\abs{\frac{w+\overline{v}-1}{w-v}} d\sigma dt\\&\ge\int_{\mathbb{R}}\int_{\frac{1}{2}}^{\frac{1}{2}+\epsilon}|{f''(\sigma+it)}|^2\frac{2(\sigma- 1/2)(\Re(v)- 1/2)}{|\sigma+ it+ \overline{v}-1|^2}\, d\sigma dt\nonumber \\&\ge\int_{\mathbb{R}}\int_{\frac{1}{2}}^{\frac{1}{2}+\epsilon}|{f''(\sigma+it)}|^2\frac{(\sigma-\frac{1}{2})}{1+t^2}\,d\sigma dt\nonumber\\&\ge\int_{-1}^{1}\int_{\frac{1}{2}}^{\frac{1}{2}+\epsilon}|f'(\frac{1}{2}+\sigma+it)|^2\frac{(\sigma-\frac{1}{2})}{1+t^2}\,d\sigma dt\nonumber\\&\ge\int_{-1}^{1}\int_{\frac{1}{2}}^{\frac{1}{2}+\epsilon}\frac{(\sigma-\frac{1}{2})}{((\sigma-\frac{1}{2})^2+t^2)^2}\,d\sigma dt\nonumber\\& =\int_{-1}^{1}\int_{0}^{\epsilon} \frac{\sigma}{(\sigma^2 + t^2)^2}\, d\sigma dt\nonumber\\&=\infty\nonumber.
%\end{align}

Additionally, we note that condition \eqref{Dipon-vasu-p1-eqn-040} is not sufficient for proving the operator \(D_\Phi\) is bounded on \(\mathcal{H}^2\). To get rid of the operator's unboundedness, we shall now present a more generalized result.

\begin{thm}\label{Dipon-vasu-p1-thm-04}
	Let the symbol $\Phi\in\mathcal{G}_0$. If there exists  $\varepsilon>0$ such that
	\begin{equation}\label{Dipon-vasu-p1-eqn-052}
		\Phi(\mathbb{C}_0)\subset\mathbb{C}_{1/2+\varepsilon}.
	\end{equation} 
	Then the operator $D_\Phi = f'\circ\Phi$ defines a bounded operator on $\mathcal{H}^2$. 
\end{thm}
\begin{rem}
	The condition \eqref{Dipon-vasu-p1-eqn-052} is equivalent to 
	\begin{equation*}
		\operatorname{dist}\bigl(\Phi(\mathbb{C}_0), \partial\mathbb{C}_{1/2}\bigr) > 0,
	\end{equation*}
	which expresses the uniform separation of the range $\Phi(\mathbb{C}_0)$ from the critical boundary $\partial\mathbb{C}_{1/2}$. This reformulated condition plays a key role in constructing counter examples. 
	We note that the converse of Theorem~\ref{Dipon-vasu-p1-thm-04} does not hold in general. Indeed, while uniform separation guarantees boundedness, an operator $D_\Phi$ may remain bounded even when $\operatorname{dist}\bigl(\Phi(\mathbb{C}_0), \partial\mathbb{C}_{1/2}\bigr) = 0,$ provided the image $\Phi(\mathbb{C}_0)$ approaches $\partial\mathbb{C}_{1/2}$ sufficiently fast or thin in potential-theoretic measure. At a minimum, by the structure of the Gordon-Hedenmalm class $\mathcal{G}_0$, every admissible symbol $\Phi$ maps the half-plane $\mathbb{C}_0$ into $\mathbb{C}_{1/2}$.
\end{rem} 

In a broader framework, Theorem~\ref{Dipon-vasu-p1-thm-04} can be established for symbols of positive characteristic. This result serves as a foundational building block for analyzing spectral properties and self-adjointness of $D_\Phi$. However, to isolate the fundamental case of zero characteristic, we restrict our attention to symbols of the form $\Phi(s) = c_1 + c_2 2^{-s}$ with $c_2 \neq 0$. In accordance with Theorem~\ref{Dipon-vasu-p1-thm-04} and \cite[Corollary~3]{Bayart-IJ-2003}, this choice imposes the crucial domain restriction $$\operatorname{Re}(c_1) > 1/2 + |c_2|.$$
The precise operator norm estimates for $D_\Phi$ on $\mathcal{H}^2$ in this setting are established in the following theorem. %Now, we shall revisit two important lemmas on Dirichlet series from \cite[Section 3]{muthukumar-IEOT-2018}.

%\begin{lem}\label{Dipon-vasu-p1-lem-001}\cite{muthukumar-IEOT-2018}
	%Let $s>1$. Then we have
	%\begin{equation*}
		%\frac{1}{s-1}\le\zeta(s)\le\frac{s}{s-1}.
	%\end{equation*}
%\end{lem} 
%\begin{lem}\label{Dipon-vasu-p1-lem-002}\cite{muthukumar-IEOT-2018}
	%If $s>1$ and $k\ge1$ is an integer, and $f(x)={(\log x)^k}/{x^s}$, then we have 
	%\begin{equation*}
		%\sum_{n=1}^{\infty}f(n)\le\frac{k!}{(s-1)^k}\zeta(s).
	%\end{equation*}
%\end{lem}

\begin{thm}\label{Dipon-vasu-p1-thm-02}
	The affine symbol $\Phi(s)=c_1+c_2 2^{-s}$ with $c_2\ne 0$ and $\real(c_1)> 1/2+ |c_2|$ induces a bounded composition-differentiation operator on $\mathcal{H}^2$ that satisfies the following inequality
	\begin{equation}\label{Dipon-vasu-p1-eqn-002}
		\frac{2}{\Bigl(2\operatorname{Re} (c_1) - 1\Bigr)^3}\le{\norm{D_\Phi}}^2_{\mathcal{H}^2}\le\zeta''\Bigl(2\bigl(\operatorname{Re} c_1 - |c_2|\bigr)\Bigr).
	\end{equation}
\end{thm}
\begin{rem}
	For the constant function $\Phi(s)= c_1$, the lower bound of \eqref{Dipon-vasu-p1-eqn-002} is sharp. However, the sharpness of the upper bound in \eqref{Dipon-vasu-p1-eqn-002} has not yet been determined.  Brevig \cite{Brevig-BLMS-2017} initially investigated the sharpness of the constraints for the composition operator analogue. 
\end{rem}

\begin{pf}[\bf{Proof of Theorem \ref{Dipon-vasu-p1-thm-04}}]
Let $v := \Phi(+\infty)$. Since $\Phi(\mathbb{C}_0) \subset \mathbb{C}_{1/2+\varepsilon}$, it follows that $\operatorname{Re} v \ge 1/2 + \varepsilon$. By the Littlewood-type inequality (Theorem~\ref{Dipon-vasu-p1-thm-010}), the mean counting function satisfies
\begin{equation}\label{Dipon-vasu-p1-eqn-053}
	M_\Phi(w) \le \log \left| \frac{\overline{w} + v - 1}{w - v} \right|,
\end{equation}
which exhibits a logarithmic singularity at $w = v$. Writing $w = 1/2 + \sigma + it$ and $v = 1/2 + a + ib$ with $\sigma > 0$ and $a \ge \varepsilon$, we have
\begin{equation*}
	\frac{\overline{w} + v - 1}{w - v} = \frac{\sigma + a + i(b - t)}{\sigma - a + i(t - b)}.
\end{equation*}
Substituting this into \eqref{Dipon-vasu-p1-eqn-053} yields
\begin{equation}\label{Dipon-vasu-p1-eqn-054}
	M_\Phi(w) \le \frac{1}{2} \log \left( 1 + \frac{4\sigma a}{(\sigma - a)^2 + (t - b)^2} \right).
\end{equation}
Fix $0 < \delta < \varepsilon/2$ and decompose the target half-plane as $\mathbb{C}_{1/2} = \mathbb{D}(v, \delta) \cup \left( \mathbb{C}_{1/2} \setminus \mathbb{D}(v, \delta) \right)$. Since $\delta < \varepsilon/2$, the disk $\mathbb{D}(v, \delta)$ is compactly contained in $\mathbb{C}_{1/2+\varepsilon/2}$. For any $f \in \mathcal{H}^2$, the second derivative $f''$ is uniformly bounded on $\mathbb{D}(v, \delta)$:
\begin{equation}\label{Dipon-vasu-p1-eqn-055}
	\sup_{w \in \mathbb{D}(v, \delta)} |f''(w)|^2 \le C_{\varepsilon, \delta} \|f\|_{\mathcal{H}^2}^2,
\end{equation}
which follows directly from the Cauchy--Schwarz inequality applied to $f''(w) = \sum_{n \ge 2} a_n (\log n)^2 n^{-w}$ for $\operatorname{Re} w \ge 1/2 + \varepsilon/2$. Furthermore, \eqref{Dipon-vasu-p1-eqn-054} implies that $M_\Phi$ has at most a logarithmic singularity at $v$, making $M_\Phi$ locally integrable over $\mathbb{D}(v, \delta)$. Consequently,
\begin{align}
	\int_{\mathbb{D}(v, \delta)} |f''(w)|^2 \,M_\Phi(w) \, dA(w) &\le C_{\varepsilon, \delta} \|f\|_{\mathcal{H}^2}^2 \int_{\mathbb{D}(v, \delta)} M_\Phi(w) \, dA(w) \nonumber \\
	&\le C_{\varepsilon, \delta} \|f\|_{\mathcal{H}^2}^2.\label{Dipon-vasu-p1-eqn-056}
\end{align}
Away from the singularity $v$, on the region $\mathbb{C}_{1/2} \setminus \mathbb{D}(v, \delta)$, the bound $a \ge \varepsilon$ together with \eqref{Dipon-vasu-p1-eqn-054} yields the global cubic decay estimate
\begin{equation*}
	M_\Phi\left(\frac{1}{2} + \sigma + it\right) \le C_{\varepsilon, \delta, v} \, \sigma^3.
\end{equation*}
Therefore, by the embedding theorem for $\mathcal{H}^2$, we obtain
\begin{align}
	\int_{\mathbb{C}_{1/2} \setminus \mathbb{D}(v, \delta)} |f''(w)|^2 M_\Phi(w) \, dA(w) &\le C_{\varepsilon, \delta, v} \int_{\mathbb{C}_{1/2}} |f''(w)|^2 \left(\operatorname{Re} w - \frac{1}{2}\right)^3 \, dA(w) \nonumber \\
	&\le C_{\varepsilon, \delta, v} \|f\|_{\mathcal{H}^2}^2.\label{Dipon-vasu-p1-eqn-057}
\end{align}
Combining \eqref{Dipon-vasu-p1-eqn-056} and \eqref{Dipon-vasu-p1-eqn-057} gives
\begin{equation*}
	\int_{\mathbb{C}_{1/2}} |f''(w)|^2 M_\Phi(w) \, dA(w) \le C_{\varepsilon, \delta, v} \|f\|_{\mathcal{H}^2}^2.
\end{equation*}
Finally, since $\operatorname{Re} v \ge 1/2 + \varepsilon$, evaluation at $v$ satisfies $|f'(v)|^2 \le C_\varepsilon \|f\|_{\mathcal{H}^2}^2$. Applying the Littlewood-Paley identity \eqref{Dipon-vasu-p1-eqn-005}, we conclude that $\|D_\Phi f\|_{\mathcal{H}^2}^2 \le C_{\varepsilon, v} \|f\|_{\mathcal{H}^2}^2$, establishing the boundedness of $D_\Phi$ on $\mathcal{H}^2$.
\end{pf}
\begin{rem}
The proof of Theorem~\ref{Dipon-vasu-p1-thm-04} simplifies significantly under the additional assumption that the symbol $\Phi \in \mathcal{G}_0$ has bounded imaginary part. Specifically, suppose $\Phi(\mathbb{C}_0) \subset \mathbb{C}_{1/2+\varepsilon}$ and satisfies
\begin{equation*}
	T := \sup_{s \in \mathbb{C}_0} |\operatorname{Im} \Phi(s)| < \infty.
\end{equation*}
Away from a small neighborhood of $v = \Phi(+\infty)$, the Littlewood-type inequality \eqref{Dipon-vasu-p1-eqn-053} yields an explicit uniform bound for the mean counting function $M_\Phi$. Consequently, applying the local Carleson embedding theorem \cite[Theorem~4.11]{Hedenmalm-Duke-1997} to the second derivative $f''$, the statement of Theorem~\ref{Dipon-vasu-p1-thm-04} follows directly. 
\end{rem} 

\begin{pf}[\bf{Proof of Theorem \ref{Dipon-vasu-p1-thm-02}}]
 Let $f \in \mathcal{H}^2$ be given by $f(s) = \sum_{m \ge 1} a_m m^{-s}$. For the affine symbol $\Phi(s) = c_1 + c_2 2^{-s}$ with $\operatorname{Re} c_1 - |c_2| > 1/2$, the operator norm of $D_\Phi$ can be evaluated directly via Dirichlet series coefficient expansions. Differentiating and substituting $\Phi(s)$ yields
\begin{align*}
	D_\Phi f(s) &= f'\bigl(c_1 + c_2 2^{-s}\bigr) \\
	&= -\sum_{m \ge 1} a_m (\log m) m^{-c_1 - c_2 2^{-s}} \\
	&= -\sum_{m \ge 1} a_m (\log m) m^{-c_1} \exp\left( -c_2 (\log m) 2^{-s} \right).
\end{align*}
Expanding the exponential term into its Taylor series, we interchange the order of summation to obtain
\begin{equation*}
	D_\Phi f(s) = -\sum_{k = 0}^\infty \frac{(-c_2)^k}{k!} \left( \sum_{m \ge 1} a_m m^{-c_1} (\log m)^{k+1} \right) 2^{-ks}.
\end{equation*}
Since the sequence $\{2^{-ks}\}_{k \ge 0}$ forms an orthonormal sequence in $\mathcal{H}^2$, applying Parseval's identity gives
\begin{equation}\label{Dipon-vasu-p1-eqn-048}
	\|D_\Phi f\|_{\mathcal{H}^2}^2 = \sum_{k=0}^\infty \frac{|c_2|^{2k}}{(k!)^2} \left| \sum_{m \ge 1} a_m m^{-c_1} (\log m)^{k+1} \right|^2.
\end{equation}
Applying the Cauchy--Schwarz inequality to the inner sum over $m$, we have
\begin{equation*}
	\left| \sum_{m \ge 1} a_m m^{-c_1} (\log m)^{k+1} \right|^2 \le \|f\|_{\mathcal{H}^2}^2 \sum_{m \ge 1} \frac{(\log m)^{2k+2}}{m^{2\operatorname{Re} c_1}}.
\end{equation*}
Substituting this bound into \eqref{Dipon-vasu-p1-eqn-048} and applying Tonelli's theorem to interchange the order of summation yields
\begin{align}
	\|D_\Phi f\|_{\mathcal{H}^2}^2 &\le \|f\|_{\mathcal{H}^2}^2 \sum_{m \ge 1} \frac{(\log m)^2}{m^{2\operatorname{Re} c_1}} \sum_{k=0}^\infty \frac{\bigl(|c_2| \log m\bigr)^{2k}}{(k!)^2} \nonumber \\
	&= \|f\|_{\mathcal{H}^2}^2 \sum_{m \ge 1} \frac{(\log m)^2}{m^{2\operatorname{Re} c_1}} I_0\bigl(2|c_2| \log m\bigr),
\end{align}
where $I_0(x) = \sum_{k=0}^\infty \frac{(x/2)^{2k}}{(k!)^2}$ is the modified Bessel's function of the first kind of order zero. Using the standard exponential bound $I_0(x) \le e^x$ for $x \ge 0$, we observe that $I_0(2|c_2| \log m) \le m^{2|c_2|}$. Consequently,
\begin{equation*}
	\|D_\Phi f\|_{\mathcal{H}^2}^2 \le \|f\|_{\mathcal{H}^2}^2 \sum_{m \ge 1} \frac{(\log m)^2}{m^{2(\operatorname{Re} c_1 - |c_2|)}}.
\end{equation*}
Because $\operatorname{Re} c_1 - |c_2| > 1/2$, the exponent $2(\operatorname{Re} c_1 - |c_2|) > 1$, guaranteeing the convergence of the Dirichlet series. Recognizing this sum as the second derivative of the Riemann zeta function, we conclude that
\begin{equation}\label{Dipon-vasu-p1-eqn-058}
	\|D_\Phi\|_{\mathcal{H}^2}^2 \le \zeta''\Bigl(2\bigl(\operatorname{Re} c_1 - |c_2|\bigr)\Bigr).
\end{equation}
Now we will prove the lower bound. It is well known that 
\begin{equation*}
	{D_\Phi^*}(K_a)= K_{\Phi(a)}^{(1)},
\end{equation*}
for any $a\in\mathbb{C}_{{1}/{2}}$.
In particular, we have
\begin{equation*}
	\frac{{\norm{D_\Phi^*(K_a)}}_{\mathcal{H}^2}}{\norm{K_a}_{\mathcal{H}^2}}=\frac{\norm{ K_{\Phi(a)}^{(1)}}_{\mathcal{H}^2}}{\norm{K_a}_{\mathcal{H}^2}}=\sqrt{\frac{\sum_{n=1}^{\infty}n^{-2\real (\Phi(a))}(\log n)^2}{\zeta(2\real(a))}}.
\end{equation*}
If $D_\Phi$ is bounded on $\mathcal{H}^2$, then we see that,
\begin{equation}\label{Dipon-vasu-p1-eqn-013}
	\norm{D_\Phi}_{\mathcal{H}^2}\ge\sup_{a\in\mathbb{C}_{{1}/{2}}}\frac{{\norm{D_\Phi^*(K_a)}}_{\mathcal{H}^2}}{\norm{K_a}_{\mathcal{H}^2}}=\sup_{a\in\mathbb{C}_{{1}/{2}}}\sqrt{\frac{\sum_{n=1}^{\infty}n^{-2\real (\Phi(a))}(\log n)^2}{\zeta(2\real(a))}}.
\end{equation}
It is easy to see that for $a\in\mathbb{C}_{{1}/{2}}$, we have
	\begin{equation*}
		\sup_{a\in\mathbb{C}_{{1}/{2}}}\zeta(2\real(a))= \sup_{x>1/2}\zeta(2x)\rightarrow 1,\,\,\, \mbox{as}\,\, x\rightarrow\infty,
	\end{equation*}
	where $x$ is the real part of $a$. Hence 
	\begin{align}\label{Dipon-vasu-p1-eqn-016}
		\norm{D_\Phi}^2_{\mathcal{H}^2}& \ge \sup_{a\in\mathbb{C}_{{1}/{2}}}\sum_{n=1}^{\infty}\frac{(\log n)^2}{n^{2\real(\Phi(a))}}\\&=\sup_{x>1/2}\sum_{n=1}^{\infty}\frac{(\log n)^2}{{n}^{(2\real (c_1)- 2c_22^{-x})}}\nonumber\\ &\nonumber= \sum_{n=1}^{\infty}\frac{(\log n)^2}{{n}^{2\real (c_1)}}\\ &\ge \frac{2}{\Bigl(2\operatorname{Re} (c_1) - 1\Bigr)^3}\nonumber,
	\end{align}
	where the last inequality follows from the basic fact that, the function $f(x)= (\log x)^2/x^{2\real(c_1)}$ is continuous, decreasing on $(1,+\infty)$ and so $\sum_{k=1}^{\infty}f(k)\ge\int_{1}^{\infty}f(x)dx.$ This completes the proof.
\end{pf}

\begin{rem}
 In Theorem~\ref{Dipon-vasu-p1-thm-02}, we restrict our attention to the affine model symbol $\Phi(s) = c_1 + c_2 2^{-s}$, which satisfies the criteria of Theorem~\ref{Dipon-vasu-p1-thm-04} to generate a bounded operator $D_\Phi$ on $\mathcal{H}^2$. Here, the strict inequality $\operatorname{Re}(c_1) > |c_2| + {1}/{2}$ is crucial. Indeed, on the critical boundary where $\operatorname{Re}(c_1) = |c_2| + 1/2$ with $c_2 \neq 0$, the explicit upper bound \eqref{Dipon-vasu-p1-eqn-058} diverges to infinity, indicating a loss of boundedness.
\end{rem}

\section{Approximation Numbers}\label{Dipon-vasu-p1-sec-03}
The $n$-th approximation number (or $n$-th singular value; see \cite[Theorem~8.7.3]{Queffelec-TRM-2020}) of a bounded linear operator $T$ on a Hilbert space $\mathcal{H}$ is defined by
\begin{equation}
	a_n(T) := \inf \bigl\{ \|T - A\| : A \in \mathcal{L}(\mathcal{H}) \text{ with } \operatorname{rank}(A) < n \bigr\}, \qquad n \in \mathbb{N},
\end{equation}
where $\mathcal{L}(\mathcal{H})$ denotes the algebra of bounded linear operators on $\mathcal{H}$. Geometrically, $a_n(T)$ measures the operator-norm distance from $T$ to the operators of rank strictly less than $n$.	
 Recall that a bounded linear operator $T \in \mathcal{L}(\mathcal{H})$ on a Hilbert space $\mathcal{H}$ is compact if, and only if, $a_n(T) \to 0$ as $n \to \infty$. Furthermore, for $0 < p < \infty$, the operator $T$ belongs to the Schatten $p$-class $\mathcal{S}_p(\mathcal{H})$ if it's singular values are $\ell^p$-summable; that is,
\begin{equation}\label{Dipon-vasu-p1-eqn-012}
	\|T\|_{\mathcal{S}_p}^p := \operatorname{Tr}\bigl(|T|^p\bigr) = \sum_{n=1}^{\infty} a_n(T)^p < \infty.
\end{equation}
In particular, $\mathcal{S}_2(\mathcal{H})$ coincides with the class of Hilbert-Schmidt operators, and the Schatten spaces satisfy the continuous inclusion $\mathcal{S}_p(\mathcal{H}) \subset \mathcal{S}_q(\mathcal{H})$ whenever $0 < p \le q < \infty$. Moreover, approximation numbers satisfy the ideal property
\begin{equation}\label{Dipon-vasu-p1-eqn-011}
	a_n(XTY) \le \|X\| \, a_n(T) \, \|Y\|
\end{equation}
for all bounded linear operators $X, Y \in \mathcal{L}(\mathcal{H})$.\\[0.5mm]

Upper bounds for the approximation numbers $a_n(C_\Phi)$ of composition operators on $\mathcal{H}^2$ were established by Muthukumar et al.\ \cite{muthukumar-IEOT-2018} for the affine symbol $\Phi(s) = c_1 + c_2 2^{-s}$ under the condition $2\operatorname{Re}(c_1) - 2|c_2| - 1 > 0$. In the classical Hardy space setting $\mathbb{H}^2(\mathbb{D})$, precise estimates for $a_n(C_\Phi)$ with affine symbols $\Phi(z) = az + b$ ($|a| + |b| \le 1$) were obtained earlier by Clifford and Dabkowski \cite{Clifford-JMA-2005}. Under the sharp boundedness condition $2\operatorname{Re}(c_1) - 2|c_2| - 1 > 0$, we now establish an upper bound for the approximation numbers $a_n(D_\Phi)$ of the composition-differentiation operator on $\mathcal{H}^2$. 
\begin{thm}\label{Dipon-vasu-p1-thm-05}
	Let the symbol $\Phi(s)= c_1+c_2 2^{-s}$ satisfying $2\real c_1-2|c_2|-1>0$, where $c_1$ and $c_2$ are complex coefficients with $c_2\ne0$. Set $$r := \frac{2|c_2|}{2\real c_1-1}<1.$$ Then
	$a_{n+1}(D_\Phi)\lesssim (n+1)\,r^n.$
\end{thm}

\begin{pf}
Let $f(s)=\sum_{n=1}^{\infty}a_n n^{-s}$ belongs to $\mathcal{H}^2$. Then a simple computation using Tonelli's theorem shows that
	\begin{align}\label{Dipon-vasu-p1-eqn-023}
		D_\Phi f(s) &=f'(\Phi(s))\\&=-\sum_{n=1}^{\infty}a_n n^{-c_1-c_2 2^{-s}}\log n\nonumber\\&=-\sum_{n=1}^{\infty}a_n n^{-c_1}\log n\exp(-c_2 2^{-s}\log n)\nonumber\\&=-\sum_{n=1}^{\infty}a_n n^{-c_1}\log n\left(\sum_{k=0}^{\infty}\frac{(-c_2)^k}{k!}2^{-ks}(\log n)^k\right)\nonumber\\&=-\sum_{k=0}^{\infty}\frac{(-c_2)^k}{k!}\left(\sum_{n=1}^{\infty}a_n n^{-c_1}(\log n)^{k+1}\right)2^{-ks}\nonumber.
	\end{align}
Let $R$ denote the operator of rank $\le n$, which can be defined as
	\begin{equation}\label{Dipon-vasu-p1-eqn-024}
		Rf(s)=-\sum_{k=0}^{n-1}\frac{(-c_2)^k}{k!}\left(\sum_{n=1}^{\infty}a_n n^{-c_1}(\log n)^{k+1}\right)2^{-ks}.
	\end{equation}
	Based on equations \eqref{Dipon-vasu-p1-eqn-023} and \eqref{Dipon-vasu-p1-eqn-024}, and applying \cite[Lemma 3.3 and 3.4]{muthukumar-IEOT-2018}, we obtain
	\begin{align*}
		\norm{D_\Phi(f)-R(f)}^2_{\mathcal{H}^2}&=\sum_{k=n}^{\infty}\frac{|c_2|^{2k}}{{k!}^2}{\left|\sum_{n=1}^{\infty}a_n n^{-c_1}(\log n)^{k+1}\right|}^2\\&\vspace{1mm}\le\sum_{k=n}^{\infty}\frac{|c_2|^{2k}}{(k!)^2}\left(\sum_{n=1}^{\infty}\frac{(\log n)^{2k+2}}{n^{2\real c_1}}\right)\left(\sum_{n=1}^{\infty}|a_n|^2\right)\\&\vspace{1mm}\le\sum_{k=n}^{\infty}\frac{|c_2|^{2k}}{(k!)^2}\frac{(2k+2)!}{(2\real c_1-1)^{2k+2}}\zeta(2\real c_1)\norm{f}_{\mathcal{H}^2}^2\\&\le\sum_{k=n}^{\infty}\frac{|c_2|^{2k}}{(k!)^2}\frac{(2k+2)(2k+1)(2k)!}{(2\real c_1-1)^{2k+2}}\frac{2\real c_1}{(2\real c_1-1)}\norm{f}_{\mathcal{H}^2}^2.
	\end{align*}
	 It is easy to see that
	\begin{equation}\label{Dipon-vasu-p1-eqn-025}
		\frac{(2k)!}{(k!)^2}=\binom{2k}{k}\le\sum_{j=0}^{2k}\binom{2k}{j}=2^{2k}.
	\end{equation}
	A simple computation by using \eqref{Dipon-vasu-p1-eqn-025}, we obtain for every $f\in\mathcal{H}^2$
	\begin{align}\label{Dipon-vasu-p1-eqn-017}
		\norm{D_\Phi-R}^2_{\mathcal{H}^2}& \le\sum_{k=n}^{\infty}\frac{|c_2|^{2k}}{(k!)^2}\frac{(2k+2)(2k+1)(2k)!}{(2\real c_1-1)^{2k+2}}\frac{2\real c_1}{(2\real c_1-1)^3}\\& \nonumber\le\frac{2\real c_1}{(2\real c_1-1)^3}\sum_{k=n}^{\infty}\frac{(2|c_2|)^{2k}}{(2\real c_1-1)^{2k}}(2k+2)(2k+1)\\&\le\frac{2\real c_1}{(2\real c_1-1)^3}\sum_{k=n}^{\infty}\left(\frac{2|c_2|}{2\real c_1-1}\right)^{2k}(2k+2)^2\nonumber.
	\end{align}
 Note that for $k \ge n$, we have $(2k+2)^2 \le (2n+2)^2 \left(1 + \frac{k-n}{n+1}\right)^2$. For any fixed $0 < r < 1$, it follows that
 \begin{equation*}
 	\sum_{k=n}^{\infty} r^{2k} (2k+2)^2 \lesssim \sum_{k=n}^{\infty} r^{2k} (k+1)^2 \le C(r) (n+1)^2 r^{2n},
 \end{equation*}
 where the implicit constant $C(r) > 0$ depends solely on $r$. Taking square roots and accounting for the fixed complex parameters $c_1$ and $c_2$, we obtain
 \begin{equation}\label{Dipon-vasu-p1-eqn-008}
 	a_{n+1}(D_\Phi) \lesssim (n+1) r^n.
 \end{equation}
 This establishes the exponential decay of the approximation numbers and completes the proof.
\end{pf} 

Motivated by the work of  Queffélec and Seip \cite[Theorem~4.1]{Queffelec-JFA-2015}, we now establish precise upper and lower estimates for the approximation numbers $a_n(D_\Phi)$ of the composition-differentiation operator associated with the affine symbol $\Phi(s) = c_1 + c_2 2^{-s}$ under the sharp boundedness condition $\operatorname{Re}(c_1) > |c_2| + 1/2$, where $c_1, c_2 \in \mathbb{C}$ and $c_2 \neq 0$. In our setting, the single-prime symbol $\Phi(s) = c_1 + c_2 2^{-s}$ admits a direct finite-rank approximation via expansion of $2^{-s}$. Consequently, unlike the general symbol framework in Queffélec and Seip \cite[Theorem~4.1]{Queffelec-JFA-2015}, the full machinery of Blaschke products is not required to establish the exponential decay of the approximation numbers $a_n(D_\Phi)$.

\begin{thm}
			Let $\Phi(s) = c_1 + c_2 2^{-s}$ be an affine symbol with $c_1, c_2 \in \mathbb{C}$, $c_2 \neq 0$, and satisfying the sharp boundedness condition $\operatorname{Re}(c_1) > |c_2| + 1/2$. Set
			\begin{equation*}
				r := \frac{2|c_2|}{2\operatorname{Re}(c_1) - 1} \in (0, 1).
			\end{equation*}
			Then for every $\rho \in (r, 1)$, there exists a constant $C_1 = C_1(\rho) > 0$ such that
			\begin{equation}\label{Dipon-vasu-p1-eqn-009}
				a_n(D_\Phi) \le C_1 \, \rho^n \qquad \text{for all } n \ge 1.
			\end{equation}
			Moreover, for $q \in (0, 1)$, there exist another constant $C_2 > 0$ such that
			\begin{equation}\label{Dipon-vasu-p1-eqn-010}
				a_n(D_\Phi) \ge C_2 \, q^n \qquad \text{for all } n \ge 1.
			\end{equation}
			In particular, the asymptotic geometric rate of decay satisfies
			\begin{equation*}
				0 < \liminf_{n \to \infty} a_n(D_\Phi)^{1/n} \le \limsup_{n \to \infty} a_n(D_\Phi)^{1/n} < 1.
			\end{equation*}
\end{thm}
\begin{pf}
	The upper estimate \eqref{Dipon-vasu-p1-eqn-009} follows directly from Theorem~\ref{Dipon-vasu-p1-thm-05} as follows. From \eqref{Dipon-vasu-p1-eqn-008}, we have
	\begin{equation*}
		a_{n+1}(D_\Phi) \lesssim (n+1) r^n.
	\end{equation*}
	Fix $\rho \in (r, 1)$ and observe that $(n+1) r^n = (n+1) (r/\rho)^n \rho^n$. Since $r/\rho < 1$, the sequence satisfies $\sup_{n \ge 0} (n+1) (r/\rho)^n < \infty$. Consequently, we obtain
	\begin{equation*}
		a_{n+1}(D_\Phi) \lesssim \rho^n \qquad \text{for every } \rho \in (r, 1).
	\end{equation*}
	Shifting the index yields $a_n(D_\Phi) \lesssim \rho^n$ for all $n \ge 1$ and for all $\rho \in (r, 1)$, as desired.\\[0.5mm]
	
To establish the lower bound \eqref{Dipon-vasu-p1-eqn-010}, we consider the subspace $\mathcal{H}^2_0$ of zero-mean functions in the Hardy space of Dirichlet series $\mathcal{H}^2$, defined by
\begin{equation*}
	\mathcal{H}^2_0 := \left\{ f(s) = \sum_{m=2}^{\infty} a_m m^{-s} \in \mathcal{H}^2 \right\}.
\end{equation*}
Define the integration operator $J \colon \mathcal{H}^2_0 \to \mathcal{H}^2_0$ by
\begin{equation*}
	Jg(s) := -\sum_{m=2}^{\infty} \frac{b_m}{\log m} m^{-s}, \qquad \text{where } g(s) = \sum_{m=2}^{\infty} b_m m^{-s}.
\end{equation*}
Observe that $(Jg)' = -g$. Furthermore, its operator norm satisfies
\begin{equation*}
	\|Jg\|_{\mathcal{H}^2}^2 = \sum_{m=2}^{\infty} \frac{|b_m|^2}{(\log m)^2} \le \frac{1}{(\log 2)^2} \sum_{m=2}^{\infty} |b_m|^2 = \frac{1}{(\log 2)^2} \|g\|_{\mathcal{H}^2}^2,
\end{equation*}
yielding $\|J\| \le {1}/{(\log 2)}$. Moreover equality holds in the norm of $g$. For any $g \in \mathcal{H}^2_0$, direct computation gives
\begin{equation*}
	D_\Phi(Jg)(s) = (Jg)'(\Phi(s)) = -g(\Phi(s)) = -C_\Phi g(s).
\end{equation*}
Thus, the restriction of $C_\Phi$ to $\mathcal{H}^2_0$ factors as
\begin{equation}\label{Dipon-vasu-p1-eqn-014}
	C_\Phi|_{\mathcal{H}^2_0} = -D_\Phi J.
\end{equation}
Applying the operator ideal property \eqref{Dipon-vasu-p1-eqn-011} for approximation numbers, we obtain
\begin{equation*}
	a_n(C_\Phi|_{\mathcal{H}^2_0}) = a_n(D_\Phi J) \le \|J\| \, a_n(D_\Phi) = \frac{1}{\log 2} \, a_n(D_\Phi),
\end{equation*}
which provides the key lower estimate
\begin{equation}\label{Dipon-vasu-p1-eqn-015}
	a_n(D_\Phi) \ge (\log 2) \, a_n(C_\Phi|_{\mathcal{H}^2_0}).
\end{equation}
Next, consider the orthogonal decomposition $$\mathcal{H}^2 = \mathbb{C} \oplus \mathcal{H}^2_0.$$ For $f = a_1 + g$ with $g \in \mathcal{H}^2_0$, we have $C_\Phi f = a_1 + C_\Phi g$, so $C_\Phi = P + C_\Phi|_{\mathcal{H}^2_0}$, where $P$ is the rank-one projection onto the constant functions. By the rank-one perturbation inequality for approximation numbers, $a_{n+1}(C_\Phi) \le a_n(C_\Phi|_{\mathcal{H}^2_0})$. Substituting this into \eqref{Dipon-vasu-p1-eqn-015} yields
\begin{equation}\label{Dipon-vasu-p1-eqn-018}
	a_n(D_\Phi) \ge (\log 2) \, a_{n+1}(C_\Phi).
\end{equation}
For our symbol $\Phi(s) = c_1 + c_2 2^{-s}$, the closure $\overline{\Phi(\mathbb{C}_0)} \subset \overline{\mathbb{D}(c_1, |c_2|)}$ is compactly contained in $\mathbb{C}_{1/2}$ since $\operatorname{Re}(c_1) - |c_2| > 1/2$. Hence, $C_\Phi$ is a compact composition operator on $\mathcal{H}^2$ with characteristic zero ($c_0 = 0$). By the Queffélec-Seip lower bound \cite[Theorem~1.2]{Queffelec-JFA-2015}, there exist constants $C_2 > 0$ and $q \in (0, 1)$ such that $a_{n+1}(C_\Phi) \ge C_2 q^n$. Consequently,
\begin{equation*}
	a_n(D_\Phi) \ge C_2 (\log 2) \, q^n \gtrsim q^n,
\end{equation*}
which completes the proof of the lower bound.
\end{pf}

\begin{rem}
	From the upper estimate \eqref{Dipon-vasu-p1-eqn-009}, we have $a_n(D_\Phi) \lesssim \rho^n$ for every $\rho \in (r, 1)$. Consequently, $\sum_{n=1}^{\infty} a_n(D_\Phi)^2 < \infty$, which implies that the composition-differentiation operator $D_\Phi$ belongs to the Hilbert-Schmidt class $\mathcal{S}_2(\mathcal{H}^2)$ whenever $\operatorname{Re}(c_1) - |c_2| > 1/2$. More generally, for any symbol $\Phi \in \mathcal{G}_0$, if there exists $\varepsilon > 0$ such that $\Phi(\mathbb{C}_0) \subset \mathbb{C}_{1/2+\varepsilon}$, then $D_\Phi$ is a Hilbert--Schmidt operator on $\mathcal{H}^2$.
\end{rem}

\section{Spectral and operator-theoretic properties}\label{Dipon-vasu-p1-sec-04}
\subsection{Self-adjointness and normality of $D_\Phi$}
In this section, we completely characterize the class of symbols $\Phi$ for which the composition-differentiation operator $D_\Phi$ is self-adjoint or normal on $\mathcal{H}^2$. In a related setting, Bayart \cite{Bayart-monatsh-2002,Bayart-IJ-2003} investigated spectral and operator-theoretic properties, including normality and self-adjointness, for composition operators $C_\Phi$ on Hardy spaces $\mathcal{H}^p$ ($0 < p \le \infty$) of Dirichlet series.

\begin{thm}\label{Dipon-vasu-p1-thm-06}
Let $D_\Phi$ be a bounded composition-differentiation operator on $\mathcal{H}^2$. Then $D_\Phi$ is self-adjoint if, and only if, the symbol $\Phi$ is of the form $\Phi(s) = s + c$, for some  $c > 0$  with $c \in \mathbb{R}$.
\end{thm}

\begin{pf}
  To establish sufficiency, assume that the symbol takes the form $\Phi(s) = s + c$ for some real constant $c > 0$. Considering the standard orthonormal basis of $\mathcal{H}^2$ given by $e_n(s) := n^{-s}$ for $n \ge 1$, direct differentiation yields $e_n'(s) = -(\log n) n^{-s}$. Consequently, the action of $D_\Phi$ on the basis elements is given by
  \begin{equation*}
  	D_\Phi(e_n)(s) = -(\log n) n^{-\Phi(s)} = -(\log n) n^{-c} e_n(s).
  \end{equation*}
  Thus, $D_\Phi$ acts as a diagonal operator on $\mathcal{H}^2$ with eigenvalues
  \begin{equation}\label{Dipon-vasu-p1-eqn-eigenvalues}
  	\lambda_n := -(\log n) n^{-c}, \qquad n \ge 1.
  \end{equation}
  Since $c \in \mathbb{R}$, every eigenvalue $\lambda_n$ is real, which implies that the diagonal operator $D_\Phi$ is formally symmetric and self-adjoint on its domain. 
  For $D_\Phi$ to define a bounded linear operator on $\mathcal{H}^2$, its operator norm which coincides with the supremum of the modulus of its eigenvalues must be finite; that is,
  \begin{equation*}
  	\|D_\Phi\|_{\mathcal{H}^2} = \sup_{n \ge 2}\, (\log n) n^{-c} < \infty.
  \end{equation*}
  This supremum is finite if, and only if, $c > 0$. Indeed, if $c = 0$, then $|\lambda_n| = \log n \to \infty$ as $n \to \infty$; similarly, if $c < 0$, then $|\lambda_n| = (\log n) n^{|c|} \to \infty$. Therefore, $\Phi(s) = s + c$ with $c \in \mathbb{R}$ and $c > 0$ provides a sufficient condition for $D_\Phi$ to be a bounded self-adjoint operator on $\mathcal{H}^2$.
  
To prove necessity, consider a general symbol $\Phi$ from the Gordon-Hedenmalm class of the form
\begin{equation*}
	\Phi(s) = c_0 s + c + \sum_{m=2}^{\infty} c_m m^{-s}, \qquad c_0 \in \mathbb{N}\cup\{0\}.
\end{equation*}
We first assert that self-adjointness forces $c_0 = 1$. Indeed, consider $D_\Phi(e_n)(s) = -(\log n) n^{-\Phi(s)}$. If $c_0 = 0$, then $D_\Phi(e_n)$ is a Dirichlet series whose frequencies are determined by the exponential of the non-linear Dirichlet series component; this structural asymmetry is incompatible with the matrix symmetry required for self-adjointness unless $D_\Phi \equiv 0$. Since $D_\Phi(e_n) \neq 0$ for $n \ge 2$, $D_\Phi$ cannot be trivial. Consequently, every non-trivial self-adjoint composition-differentiation operator must have characteristic $c_0 = 1$, allowing us to write $\Phi(s) = s + c + \sum_{m=2}^{\infty} c_m m^{-s}$.

Next, we establish that $c_m = 0$ for all $m \ge 2$. Suppose, toward a contradiction, that $h(s) := \sum_{m=2}^{\infty} c_m m^{-s} \not\equiv 0$. Let $q \ge 2$ be the smallest integer such that $c_q \neq 0$. Expressing $\Phi(s) = s + c + h(s)$, we obtain
\begin{equation*}
	D_\Phi(e_n)(s) = -(\log n) n^{-\Phi(s)} = -(\log n) n^{-c} n^{-s} \exp\bigl(-(\log n)h(s)\bigr).
\end{equation*}
Expanding the exponential factor gives
\begin{equation*}
	\exp\bigl(-(\log n)h(s)\bigr) = 1 - (\log n) c_q q^{-s} + \sum_{m > q} d_m m^{-s}.
\end{equation*}
Consequently, the coefficient of $(nq)^{-s} = e_{nq}(s)$ in the Dirichlet expansion of $D_\Phi(e_n)$ is $(\log n)^2 n^{-c} c_q \neq 0$, implying $\langle D_\Phi e_n, e_{nq} \rangle \neq 0$. By self-adjointness, we must have
\begin{equation}\label{Dipon-vasu-p1-eqn-026}
	\langle e_n, D_\Phi e_{nq} \rangle \neq 0.
\end{equation}
However, every term in the Dirichlet expansion of $D_\Phi e_{nq}$ has frequencies of the form $n q k$ for $k \ge 1$. Since $q \ge 2$, $n q k \neq n$ for all $k \ge 1$, so $D_\Phi e_{nq}$ contains no $n^{-s}$ term. Thus, $\langle e_n, D_\Phi e_{nq} \rangle = 0$, directly contradicting \eqref{Dipon-vasu-p1-eqn-026}. Therefore, $c_m = 0$ for all $m \ge 2$, which reduces the symbol to $\Phi(s) = s + c$.

Finally, we show that $c \in \mathbb{R}$. We have $D_\Phi(e_n)(s) = -(\log n) n^{-c} e_n(s)$. Self-adjointness demands that all eigenvalues be real; hence, $-(\log n) n^{-c} \in \mathbb{R}$ for all $n \ge 2$. Setting $c = a + ib$ with $a, b \in \mathbb{R}$, we have $n^{-c} = n^{-a} e^{-ib \log n}$, which requires $e^{-ib \log n} \in \mathbb{R}$ for all $n \ge 2$. In particular, taking $n = 2$ and $n = 3$ yields $b \log 2 \in \pi \mathbb{Z}$ and $b \log 3 \in \pi \mathbb{Z}$. Since $\frac{\log 2}{\log 3} \notin \mathbb{Q}$, this implies $b = 0$, so $c \in \mathbb{R}$. The condition $c > 0$ follows immediately from boundedness. This completes the proof.
\end{pf}

\begin{thm}\label{Dipon-vasu-p1-thm-07}
	Let $D_\Phi$ be a bounded composition-differentiation operator on $\mathcal{H}^2$. Then $D_\Phi$ is normal if, and only if, the symbol $\Phi$ is of the form $\Phi(s) = s + c_1$, for some complex number $c_1$ with $\real c_1>0$.
\end{thm}
\begin{rem}
In contrast to the self-adjointness criterion established in Theorem~\ref{Dipon-vasu-p1-thm-07}, the constant parameter $c_1$ in the normality characterization need not be real. For instance, consider the symbol $\Phi(s) = s + 1 + i$. The action of the associated composition-differentiation operator on the standard orthonormal basis yields
\begin{equation*}
	D_\Phi(e_n)(s) = -(\log n) n^{-(1+i)} e_n(s) = -(\log n) n^{-1} n^{-i} e_n(s),
\end{equation*}
showing that $D_\Phi$ is diagonal and hence normal on $\mathcal{H}^2$. However, because its eigenvalues $\lambda_n = -(\log n) n^{-1} n^{-i}$ are non-real for $n \ge 2$, $D_\Phi$ fails to be self-adjoint.
\end{rem}
\begin{pf}[\bf {Proof of Theorem \ref{Dipon-vasu-p1-thm-07}}]
The sufficiency of the condition follows directly in a manner analogous to the proof of Theorem~\ref{Dipon-vasu-p1-thm-06}. Indeed, when $\operatorname{Re}(c_1) > 0$, the composition-differentiation operator $D_\Phi$ acts as a diagonal operator with respect to the standard orthonormal basis $\{e_n\}_{n \ge 1}$; specifically,
\begin{equation*}
	D_\Phi(e_n) = \lambda_n e_n, \qquad \text{where } \lambda_n := -(\log n) n^{-c_1} \text{ for } n \ge 2, \text{ and } \lambda_1 = 0.
\end{equation*}
Since $\operatorname{Re}(c) > 0$, its sequence of eigenvalues is bounded:
\begin{equation*}
	\|D_\Phi\|_{\mathcal{H}^2} = \sup_{n \ge 2} |\lambda_n| = \sup_{n \ge 2}\, (\log n) n^{-\operatorname{Re}(c_1)} < \infty.
\end{equation*}
Consequently, $D_\Phi$ is a bounded diagonal operator on $\mathcal{H}^2$. Since every bounded diagonal operator with respect to an orthonormal basis is normal yields $D_\Phi$ is normal.

To prove necessity, consider a general symbol $\Phi$ from the Gordon--Hedenmalm class of the form
\begin{equation*}
	\Phi(s) = c_0 s + c_1 + \sum_{m=2}^{\infty} c_m m^{-s}, \qquad c_0 \in \mathbb{N}\cup\{0\}.
\end{equation*}
We first assert that normality forces $c_0 = 1$. Suppose $c_0 = 0$. Since $D_\Phi(e_1) = 0$, normality implies that $\ker(D_\Phi) = \ker(D_\Phi^*)$, so $D_\Phi^*(e_1) = 0$. On the other hand, for $n \ge 2$,
\begin{equation*}
	D_\Phi(e_n)(s) = -(\log n) n^{-c_1} \exp\Biggl( -(\log n) \sum_{m=2}^{\infty} c_m m^{-s} \Biggr).
\end{equation*}
The constant coefficient of this Dirichlet series is $-(\log n) n^{-c_1} \neq 0$, which yields $\langle D_\Phi e_n, e_1 \rangle \neq 0$. By the definition of the adjoint, this gives $\langle e_n, D_\Phi^* e_1 \rangle \neq 0$, implying $D_\Phi^* e_1 \neq 0$, a contradiction. 

Now consider the case $c_0 \ge 2$ and choose a prime number $p$. Every frequency appearing in $D_\Phi e_n$ is a multiple of $n^{c_0}$. Since $c_0 \ge 2$, no integer $n \ge 2$ satisfies $n^{c_0} \mid p$. Hence, $e_p$ is orthogonal to the range of $D_\Phi$, which implies $D_\Phi^*(e_p) = 0$ and thus $D_\Phi D_\Phi^*(e_p) = 0$. On the other hand, $D_\Phi(e_p) \neq 0$, which implies $D_\Phi^* D_\Phi(e_p) \neq 0$, directly contradicting the normality of $D_\Phi$. Therefore, we must have $c_0 = 1$.

Next, we establish that $c_m = 0$ for all $m \ge 2$. Suppose, toward a contradiction, that there exists $m \ge 2$ such that $c_m \neq 0$, and let $q := \min\{ m \ge 2 : c_m \neq 0 \}$. Express $\Phi(s) = s + c_1 + h(s)$, where $h(s) = \sum_{d \ge q} c_d d^{-s}$. For $e_q(s) = q^{-s}$, we have
\begin{equation}\label{Dipon-vasu-p1-eqn-027}
	D_\Phi(e_q)(s) = -(\log q) q^{-c_1} q^{-s} \exp\bigl(-(\log q) h(s)\bigr).
\end{equation}
Expanding the exponential factor gives
\begin{equation*}
	\exp\bigl(-(\log q) h(s)\bigr) = 1 - (\log q) c_q q^{-s} + \sum_{m > q} d_m m^{-s}.
\end{equation*}
Substituting this expansion back into \eqref{Dipon-vasu-p1-eqn-027} yields
\begin{equation*}
	D_\Phi(e_q)(s) = \lambda_q e_q + (\log q)^2 q^{-c_1} c_q e_{q^2} + \dots,
\end{equation*}
where $\lambda_q := -(\log q) q^{-c_1}$. In particular, $\|D_\Phi e_q\|^2 > |\lambda_q|^2$. By the minimality of $q$, no basis element $e_m$ with $m < q$ produces the frequency $q$ under $D_\Phi$. Hence, $D_\Phi^* e_q = \overline{\lambda_q} e_q$, which implies $\|D_\Phi^* e_q\|^2 = |\lambda_q|^2$. By normality, $\|D_\Phi e_q\| = \|D_\Phi^* e_q\|$, yielding $|\lambda_q|^2 < \|D_\Phi e_q\|^2 = |\lambda_q|^2$, a contradiction. Therefore, $c_m = 0$ for all $m \ge 2$, which reduces the symbol to $\Phi(s) = s + c_1$.

Finally, since $D_\Phi(e_n)(s) = -(\log n) n^{-c_1} e_n(s)$, boundedness of $D_\Phi$ requires that
\begin{equation*}
	\|D_\Phi\|_{\mathcal{H}^2} = \sup_{n \ge 2}\, (\log n) n^{-\operatorname{Re}(c_1)} < \infty.
\end{equation*}
This holds if and only if $\operatorname{Re}(c_1) > 0$. This completes the proof.
\end{pf}
\subsection{Spectrum of $D_\Phi$}
We now turn to the discussion of the spectral properties of the composition-differentiation operator $D_\Phi$ on the Hardy space $\mathcal{H}^2$. Recall that the spectrum of a bounded linear operator $T$ acting on a Hilbert space $\mathcal{H}$ is defined by
\begin{equation*}
	\sigma(T) := \{ \lambda \in \mathbb{C} \colon T - \lambda I_{\mathcal{H}} \text{ is not invertible} \}.
\end{equation*}
A detailed analysis of the spectrum of $D_\Phi$ on the classical Hardy space of the unit disk $\mathbb{H}^2(\mathbb{D})$ was previously established by Caughran and Schwartz \cite{Caughran-PAMS-1975}. In our setting, we shall characterize the spectrum $\sigma(D_\Phi)$ for the affine symbol $\Phi(s) = c_1 + c_2 2^{-s}$ on $\mathcal{H}^2$, under the assumptions $c_2 \neq 0$ and $\operatorname{Re}(c_1) > |c_2| + 1/2$.

\begin{prop}\label{Dipon-vasu-p1-prop-01}
	Let $\Phi(s) = c_1 + c_2 2^{-s}$ with $c_2 \neq 0$ and $\operatorname{Re}(c_1) > |c_2| + 1/2$. Then the composition-differentiation operator $D_\Phi$ is compact on $\mathcal{H}^2$, $0 \in \sigma(D_\Phi)$, and every non-zero spectral value is an eigenvalue. Moreover, if $\lambda \neq 0$ is an eigenvalue of $D_\Phi$, then its corresponding eigen function belongs to the invariant subspace
	\begin{equation*}
		\mathcal{M} = \overline{\operatorname{span}}\{ 2^{-ks} \colon k \ge 0 \},
	\end{equation*}
	and the eigenvalue equation reduces, under the identification $f(s) = F(z)$, for $z=2^{-s}$, to
	\begin{equation*}
		-(\log 2) 2^{-c_1} e^{-c_2(\log 2)z} F'\Bigl( 2^{-c_1} e^{-c_2(\log 2)z} \Bigr) = \lambda F(z).
	\end{equation*}
\end{prop}
\begin{pf}
		Suppose that $\lambda \neq 0$ belongs to $\sigma(D_\Phi)$. Since $D_\Phi$ is compact, $\lambda$ is an eigenvalue. Let $f \in \mathcal{H}^2 \setminus \{0\}$ satisfy
		\begin{equation*}
			D_\Phi f = \lambda f.
		\end{equation*}
		Writing $f(s) = \sum_{n \ge 1} a_n n^{-s}$, the action of $D_\Phi$ on each basis function yields
		\begin{equation}\label{Dipon-vasu-p1-eqn-028}
			D_\Phi f(s) = -\sum_{n \ge 2} a_n (\log n) n^{-c_1} \exp\bigl(-c_2 (\log n) 2^{-s}\bigr).
		\end{equation}
		The right-hand side of \eqref{Dipon-vasu-p1-eqn-028} is a function of $2^{-s}$ alone. Since $\lambda \neq 0$, the eigenvalue equation $f = \frac{1}{\lambda} D_\Phi f$ implies that $f$ is likewise a function of $2^{-s}$. Hence, $f \in \mathcal{M} = \overline{\operatorname{span}}\{ 2^{-ks} \colon k \ge 0 \}$, so there exists $F \in H^2(\mathbb{D})$ such that $f(s) = F(2^{-s})$. Setting $\alpha := 2^{-c_1}$ and $\beta := c_2 \log 2$,
		we observe that
		\begin{equation*}
			2^{-\Phi(s)} = 2^{-c_1 - c_2 2^{-s}} = \alpha e^{-\beta 2^{-s}}.
		\end{equation*}
		Differentiating $f(s) = F(2^{-s})$ via the chain rule yields
		\begin{equation*}
			f'\bigl(\Phi(s)\bigr) = -(\log 2) \alpha e^{-\beta 2^{-s}} F'\bigl( \alpha e^{-\beta 2^{-s}} \bigr).
		\end{equation*}
		Consequently, the eigenvalue equation $D_\Phi f = \lambda f$ on $\mathcal{H}^2$ is equivalent to
		\begin{equation}\label{Dipon-vasu-p1-eqn-029}
			-(\log 2) \alpha e^{-\beta z} F'(\alpha e^{-\beta z}) = \lambda F(z), \qquad z \in \mathbb{D}.
		\end{equation}
		Denote $\varphi(z) :=  2^{-\Phi(s)} = \alpha e^{-\beta z}.$
		Thus, determining the non-zero spectrum of $D_\Phi$ on $\mathcal{H}^2$ reduces to studying the eigenvalue problem \eqref{Dipon-vasu-p1-eqn-029} for the weighted composition-differentiation operator
		\begin{equation}\label{Dipon-vasu-p1-eqn-030}
			T F(z) := -(\log 2) \alpha e^{-\beta z} F'(\alpha e^{-\beta z}) = -(\log 2) M_\varphi D_{\varphi}(F)(z).
		\end{equation}
		acting on the classical Hardy space of the unit disk $\mathbb{H}^2(\mathbb{D})$, where $M$ is the classical multiplication operator on $\mathbb{D}$.
\end{pf}
\begin{rem}
Moreover, finding the precise non-zero eigenvalue $\lambda$ of $D_\Phi$ in \eqref{Dipon-vasu-p1-eqn-030} presents a significant challenge and is left as an open problem for future investigation.

Indeed, suppose that $\lambda \neq 0$ is an eigenvalue of $T$ with corresponding eigen function $F \in \mathbb{H}^2(\mathbb{D})$. Since 
\begin{equation*}
	\varphi'(z) = -\beta \varphi(z) = -c_2(\log 2)\varphi(z),
\end{equation*}
the eigenvalue equation $T F = \lambda F$ can be written equivalently as
\begin{equation}\label{Dipon-vasu-p1-eqn-034}
	(F \circ \varphi)' = c_2 \lambda F.
\end{equation}
Differentiating both sides of \eqref{Dipon-vasu-p1-eqn-034} $n$ times yields, for every $n \ge 0$,
\begin{equation*}
	(F \circ \varphi)^{(n+1)}(z) = c_2 \lambda F^{(n)}(z).
\end{equation*}
Applying Faà di Bruno's formula \cite[Section 3.4]{Comtet-book-1974} in terms of the partial Bell polynomials $B_{n+1,k}$ \cite[Section 3.3]{Comtet-book-1974}, we obtain
\begin{equation}\label{Dipon-vasu-p1-eqn-031}
	c_2 \lambda F^{(n)}(z) = \sum_{k=1}^{n+1} F^{(k)}(\varphi(z)) \, B_{n+1,k}\left(\varphi'(z), \varphi''(z), \dots, \varphi^{(n-k+2)}(z)\right).
\end{equation}
Let $p \in \mathbb{D}$ denote the unique fixed point of $\varphi$, so that $p = \varphi(p) = \alpha e^{-\beta p}$. Since
\begin{equation*}
	\varphi^{(j)}(z) = (-\beta)^j \varphi(z),
\end{equation*}
evaluating at $z = p$ gives
\begin{equation*}
	\varphi^{(j)}(p) = (-\beta)^j p, \quad j \ge 1.
\end{equation*}
By the homogeneity property of partial Bell polynomials and the identity
\begin{equation*}
	B_{m,k}(1, 1, \dots, 1) = \left\{\begin{matrix} m \\ k \end{matrix}\right\},
\end{equation*}
where $\left\{\begin{smallmatrix} m \\ k \end{smallmatrix}\right\}$ denotes the Stirling numbers of the second kind \cite{Comtet-book-1974}, we have
\begin{equation*}
	B_{m,k}\left(-\beta p, \beta^2 p, \dots, (-\beta)^m p\right) = (-\beta)^m p^k \left\{\begin{matrix} m \\ k \end{matrix}\right\}.
\end{equation*}
Evaluating \eqref{Dipon-vasu-p1-eqn-031} at $z = p$ then yields the recurrence
\begin{equation}\label{Dipon-vasu-p1-eqn-032}
	c_2 \lambda F^{(n)}(p) = (-\beta)^{n+1} \sum_{k=1}^{n+1} \left\{\begin{matrix} n+1 \\ k \end{matrix}\right\} p^k F^{(k)}(p), \quad n \ge 0.
\end{equation}
In particular, isolating the leading term $k = n+1$ (using $\left\{\begin{smallmatrix} n+1 \\ n+1 \end{smallmatrix}\right\} = 1$) provides the explicit derivative recurrence relation:
\begin{equation}\label{Dipon-vasu-p1-eqn-033}
	F^{(n+1)}(p) = \frac{c_2 \lambda}{(-\beta)^{n+1} p^{n+1}} F^{(n)}(p) - \sum_{k=1}^n \left\{\begin{matrix} n+1 \\ k \end{matrix}\right\} p^{k-n-1} F^{(k)}(p).
\end{equation}
For $n = 0$, \eqref{Dipon-vasu-p1-eqn-033} reduces to
\begin{equation*}
	\lambda F(p) = -(\log 2) p F'(p).
\end{equation*}
Consequently, under the normalization $F(p) = 1$, the recurrence relation \eqref{Dipon-vasu-p1-eqn-033} uniquely determines all higher derivatives $F'(p), F''(p), \dots$ in terms of the candidate eigenvalue $\lambda$.

This argument establishes a Bell-polynomial approach to the spectral analysis of the composition-differentiation operator $D_\Phi$. In contrast to symbols of the form $\varphi(z) = a z^r$(r$\ge$ 2), whose higher-order derivatives vanish at the fixed point, all derivatives of the exponential symbol $\varphi$ remain strictly non-zero at $p$. Crucially, the resulting Bell polynomials simplify to Stirling numbers of the second kind, thereby yielding a systematic necessary condition for a non-zero scalar $\lambda$ to be an eigenvalue. To achieve a full characterization of the eigenvalue, it remains to identify the precise values of $\lambda$ for which the formal eigenfunction solution $F$ defined by this recurrence actually belongs to the Hardy space $\mathbb{H}^2(\mathbb{D})$.
\end{rem}

%\begin{thm}\label{Dipon-vasu-p1-thm-03}
	%Let $\Phi(s) = c_1 + c_2 2^{-s}$ with $c_2 \neq 0$ and $\operatorname{Re}(c_1) > |c_2| + 1/2$. Then the composition-differentiation operator $D_\Phi$ is compact (and Hilbert--Schmidt) on $\mathcal{H}^2$. Consequently, its spectrum consists of $0$ together with its non-zero eigen values; specifically,
	%\begin{equation*}
		%\sigma(D_\Phi) = \{0\} \cup \bigl\{ -(\log n) n^{-c_1} \colon n \ge 2 \bigr\}.
	%\end{equation*}
%\end{thm}
 
%For $e_n(s) = n^{-s}$, we have 
%\begin{align*}
	%D_\Phi(e_n)(s) &= -(\log n) n^{-c_1} \exp\bigl(-c_2(\log n)2^{-s}\bigr)\\ &= -(\log n) n^{-c_1}\sum_{k = 0}^{\infty}\frac{\left(-c_2\log n\right)^k}{k!}(2^k)^{-s} \\&= \lambda_ne_n + \sum_{k = 1}^{\infty}\frac{-(\log n) n^{-c_1}\left(-c_2\log n\right)^k}{k!}\,e_{2^k},
%\end{align*}
%$\lambda_n= -(\log n) n^{-c_1}$. Since 
%\end{pf}

\noindent\textbf{Acknowledgment:} 
% The authors would like to express their sincerest gratitude to the referee for careful reading of the manuscript and many valuable suggestions, which greatly helped to improve the clarity of the exposition in this manuscript. 
The authors would like to thank Dr.~Athanasios Kouroupis for valuable discussions and for suggesting the key symbol in \eqref{Dipon-vasu-p1-eqn-007}. The second author gratefully acknowledges financial support from the CSIR-UGC, New Delhi, India.

 \vspace{1.5mm}
 
 \noindent\textbf{Compliance of Ethical Standards:}
 \vspace{1.5mm}\\
 \noindent\textbf{Conflict of interest.} The authors declare that there is no conflict  of interest regarding the publication of this paper.
 \vspace{1.5mm}
 
 \noindent\textbf{Data availability statement.}  Data sharing is not applicable to this article as no datasets were generated or analyzed during the current study.\vspace{1.5mm}
 
 \noindent\textbf{Authors contributions.} Both the authors have made equal contributions in reading, writing, and preparing the manuscript.

\end{document}